\documentclass[reqno,a4paper]{amsart}
\usepackage[british]{babel}

\usepackage{graphicx, amsmath, amsfonts, amssymb, mathrsfs, amsthm, tikz, tikz-cd, amscd, amsbsy, amstext, mathtools, enumerate, enumitem, stmaryrd, tensor, a4wide, fancyhdr, ifthen, empheq, verbatim, latexsym, epigraph,tabularray, xfrac}

\usepackage[parfill]{parskip}

\usepackage[scr=euler]{mathalfa}

\mathtoolsset{showonlyrefs,showmanualtags}

\usepackage{etoolbox}

\usepackage{mathpazo}
\usepackage{needspace}
\usepackage{microtype}

\usepackage[T1]{fontenc}

\usepackage{nameref,zref-xr} 
\zxrsetup{toltxlabel} 
\zexternaldocument*[foundations-]{COHA_Foundations} 
\zexternaldocument*[COHA-Yangian-]{COHAs_Yangians} 
\zexternaldocument*[torsion-pairs-]{Torsion_pairs_final} 

\usetikzlibrary{positioning,shapes,shadows,arrows}

\tikzstyle{roundbox} = [rectangle, draw, text centered, rounded corners,
                        minimum height=2em]
\tikzstyle{connector} = [draw, -latex']

\usepackage{color}
\definecolor{MyDarkBlue}{rgb}{0.15,0.25,0.45}

\usepackage[toc,page]{appendix}
\allowdisplaybreaks

\newif\ifpersonal
\personaltrue 
\ifpersonal

\newcommand{\todo}[1]{\textcolor{red}{(Todo: #1)}}
\else
\newcommand\ast{\personal}[1]{\ignorespaces}
\newcommand\ast{\parth}[1]{\ignorespaces}
\newcommand\ast{\todo}[1]{\ignorespaces}
\fi

\newcommand{\calA}{\mathcal{A}}

\newcommand{\calC}{\mathcal{C}}
\newcommand{\calD}{\mathcal{D}}

\newcommand{\calF}{\mathcal{F}}
\newcommand{\calG}{\mathcal{G}}

\newcommand{\calI}{\mathcal{I}}

\newcommand{\calL}{\mathcal{L}}

\newcommand{\calO}{\mathcal{O}}
\newcommand{\calP}{\mathcal{P}}

\newcommand{\calS}{\mathcal{S}}

\newcommand{\calX}{\mathcal{X}}

\newcommand{\A}{\mathbb{A}}
\newcommand{\C}{\mathbb{C}}
\newcommand{\D}{\mathbb{D}}

\newcommand{\G}{\mathbb{G}}
\newcommand{\HH}{\mathbb{H}}

\newcommand{\J}{\mathbb{J}}

\newcommand{\LL}{\mathbb{L}}
\newcommand{\N}{\mathbb{N}}
\newcommand{\PP}{\mathbb{P}}
\newcommand{\Q}{\mathbb{Q}}
\newcommand{\R}{\mathbb{R}}

\newcommand{\Y}{\mathbb{Y}}
\newcommand{\Z}{\mathbb{Z}}

\newcommand{\scrA}{\mathscr{A}}
\newcommand{\scrB}{\mathscr{B}}
\newcommand{\scrC}{\mathscr{C}}
\newcommand{\scrD}{\mathscr{D}}

\newcommand{\scrF}{\mathscr{F}}

\newcommand{\scrO}{\mathscr{O}}

\newcommand{\scrT}{\mathscr{T}}

\newcommand{\sfA}{\mathsf{A}}

\newcommand{\sfC}{\mathsf{C}}
\newcommand{\sfD}{\mathsf{D}}
\newcommand{\sfE}{\mathsf{E}}

\newcommand{\sfH}{\mathsf{H}}

\newcommand{\sfK}{\mathsf{K}}

\newcommand{\sfU}{\mathsf{U}}

\newcommand{\sff}{\mathsf{f}}

\newcommand{\frakS}{\mathfrak{S}}

\newcommand{\frakg}{\mathfrak{g}}
\newcommand{\frakh}{\mathfrak{h}}

\newcommand{\frakn}{\mathfrak{n}}

\newcommand{\frakr}{\mathfrak{r}}

\newcommand{\bfH}{\mathbf{H}}

\newcommand{\bfT}{\mathbf{T}}

\newcommand{\bfY}{\mathbf{Y}}
\newcommand{\bfX}{\mathbf{X}}

\newcommand{\bfd}{\mathbf{d}}

\newcommand{\bfv}{\mathbf{v}}

\newcommand{\half}{1/2}

\newcommand{\catMod}{\mathsf{Mod}}
\newcommand{\catmod}{\mathsf{mod}}

\newcommand{\catP}{\mathsf{P}}

\newcommand{\catPC}{\mathsf{P}_{C}}
\newcommand{\barcatPC}{\overline{\mathsf{P}}_{C}}

\newcommand{\Stab}{\mathsf{Stab}}

\newcommand{\Mod}{\textrm{-} \mathsf{Mod}}

\newcommand{\modPi}{\catmod\, \Pi}
\newcommand{\ModPi}{\catMod\, \Pi}
\newcommand{\nilpPi}{\mathsf{nilp}\, \Pi}
\newcommand{\nilp}{\mathsf{nilp}}

\newcommand{\catCoh}{\mathsf{Coh}}
\newcommand{\barcatCoh}{\overline{\mathsf{Coh}}}

\newcommand{\catCohC}{\mathsf{Coh}_C}

\newcommand{\catDb}{\mathsf{D}^\mathsf{b}}
\newcommand{\catDbC}{\mathsf{D}^\mathsf{b}_C}

\newcommand{\kdLambdaqv}{\tensor*[^k]{\mathbf{\Lambda}}{_\qv}}

\newcommand{\dstackPerfps}{\mathbf{Perf}_{\mathsf{ps}}}
\newcommand{\dstackCoh}{\mathbf{Coh}}

\newcommand{\dstackCohps}{\mathbf{Coh}_{\mathsf{ps}}}

\newcommand{\dstackRep}{\mathbf{Rep}}
\newcommand{\stackRep}{\mathfrak{Rep}}

\newcommand{\dLambda}{\mathbf{\Lambda}}

\newcommand{\qv}{\mathcal{Q}}
\newcommand{\qvfin}{\mathcal{Q}_\sff}
\newcommand{\doubleqv}{\overline{\qv}}

\newcommand{\Lalpha}{\check{\alpha}}
\newcommand{\Lomega}{\check{\omega}}
\newcommand{\Ltheta}{\check{\theta}}

\newcommand{\Llambda}{{\check{\lambda}}}

\newcommand{\Lrho}{\check{\rho}}

\newcommand{\rootset}{\Delta}
\newcommand{\rootsetre}{\rootset^{\mathsf{re}}}
\newcommand{\rootsetim}{\rootset^{\mathsf{im}}}
\newcommand{\rootsetfin}{\rootset_\sff}

\newcommand{\rootlattice}{\bfY}
\newcommand{\corootlattice}{\check{\bfY}}
\newcommand{\rootlatticefin}{\rootlattice_\sff}
\newcommand{\corootlatticefin}{\check{\rootlattice}_\sff}

\newcommand{\weightlattice}{\bfX}
\newcommand{\coweightlattice}{\check{\weightlattice}}

\newcommand{\coweightlatticefin}{\coweightlattice_\sff}

\newcommand{\frakgfin}{\frakg_\sff}

\newcommand{\fraknfin}{\frakn_\sff}

\newcommand{\sfex}{\mathsf{ex}}

\newcommand{\rsv}{X}
\newcommand{\ssv}{X_\emptyset}
\newcommand{\psv}{X_J}

\newcommand{\Jcomp}{{J^\mathsf{c}}}

\newcommand{\frakgell}{\mathfrak{g}_{\mathsf{ell}}}
\newcommand{\fraknell}{\mathfrak{n}_{\mathsf{ell}}}
\newcommand{\fraknellLlambda}{\mathfrak{n}_{\mathsf{ell}, \Llambda}}
\newcommand{\fraknellJ}{\mathfrak{n}_{\mathsf{ell}, J}}

\newcommand{\zeroeY}{\mathbb{Y}^{\mathfrak{e},\, 0}}

\newcommand{\gr}{\mathsf{gr}}

\newcommand{\ad}{\mathsf{ad}}

\newcommand{\Hbullet}{\mathsf{H}^\bullet}
\newcommand{\HBMbullet}{\mathsf{H}_\bullet^{\mathsf{BM}}}

\newcommand{\HBMbulletA}{\mathsf{H}_\bullet^A}

\newcommand{\coha}{\mathbf{HA}}

\newcommand{\cohaqv}{\coha_\qv}

\newcommand{\id}{{\mathsf{id}}}

\newcommand{\ev}{\mathsf{ev}}

\newcommand{\Hom}{\mathsf{Hom}}
\newcommand{\calHom}{\mathcal{H}\mathsf{om}}

\newcommand{\calEnd}{\mathcal{E}\mathsf{nd}}
\newcommand{\Aut}{\mathsf{Aut}}
\newcommand{\End}{\mathsf{End}}

\newcommand{\Pic}{\mathsf{Pic}}
\newcommand{\GL}{\mathsf{GL}}
\newcommand{\SL}{\mathsf{SL}}

\newcommand{\pr}{\mathsf{pr}}

\makeatletter
\newcommand{\colim@}{%
	\vtop{\m@th\ialign{##\cr
			\hfil$\operator@font colim$\hfil\cr
			\noalign{\nointerlineskip\kern1.5\ex@}\cr
			\noalign{\nointerlineskip\kern-\ex@}\cr}}%
}
\newcommand{\colim}{%
	\mathop{\mathpalette\colim@{\textstyle}}\nmlimits@
}
\makeatother

\usepackage{thmtools}

\declaretheoremstyle[
    spaceabove=1em, spacebelow=1em,
    headfont=\bfseries, notefont=\normalfont, bodyfont=\itshape,
    headpunct={.}, notebraces={}{}, postheadspace={ }
    ]{basic-theorem}
\declaretheoremstyle[
    spaceabove=1em, spacebelow=1em,
    headfont=\bfseries, notefont=\normalfont, bodyfont=\normalfont,
    headpunct={.}, notebraces={(}{)}, postheadspace={ }, qed=$\oslash$,
    ]{basic-definition}
\declaretheoremstyle[
    spaceabove=1em, spacebelow=1em,
    headfont=\itshape, notefont=\normalfont, bodyfont=\normalfont,
    headpunct={.}, notebraces={(}{)}, postheadspace={ }, qed=$\triangle$,
    ]{basic-remark}

\theoremstyle{basic-theorem}

\newtheorem{theorem}{Theorem}[section]
\newtheorem{corollary}[theorem]{Corollary}
\newtheorem{lemma}[theorem]{Lemma}
\newtheorem{proposition}[theorem]{Proposition}

\declaretheorem[style=basic-definition, numberlike=theorem]{definition}

\declaretheorem[style=basic-definition, numberlike=theorem]{notation}

\declaretheorem[style=basic-remark, numberlike=theorem]{remark}



\numberwithin{equation}{section}
 
 
 
\newtheorem{theoremintroduction}{Theorem}

\usepackage{hyperref}

\hypersetup{anchorcolor=black, linktocpage=true,
	colorlinks=true,
	citecolor=MyDarkBlue,
	linkcolor=MyDarkBlue,
	urlcolor=black,
	pdftitle={Kleinian orbifolds, Cohomological Hall Algebras, and Yangians},
	breaklinks=true,
	plainpages=true
}
\usepackage[hyperpageref]{backref}

\title[Kleinian orbifolds, Cohomological Hall Algebras, and Yangians]{Kleinian orbifolds, Cohomological Hall Algebras, and Yangians}

\author[F.~Sala]{Francesco Sala}
\address[Francesco Sala]{Università di Pisa, Dipartimento di Matematica, Largo Bruno Pontecorvo 5, 56127 Pisa (PI), Italy}
\address{Kavli IPMU (WPI), UTIAS, The University of Tokyo, Kashiwa, Chiba 277-8583, Japan}
\email{\href{mailto:francesco.sala@unipi.it}{francesco.sala@unipi.it}}

\author[O.~Schiffmann]{Olivier Schiffmann}
\address[Olivier Schiffmann]{Laboratoire de Mathématiques d'Orsay, Université de Paris-Sud Paris-Saclay, B\^at. 425, 91405 Orsay Cedex, France, UMR8628 (CNRS)}
\address{Simion Stoilow Institute of Mathematics, Bucharest, Romania}
\email{\href{mailto:olivier.schiffmann@universite-paris-saclay.fr}{olivier.schiffmann@universite-paris-saclay.fr}}

\author[P.~Shimpi]{Parth Shimpi}
\address[Parth Shimpi]{The Mathematics and Statistics Building, University of Glasgow, University Place, Glasgow G12 8QQ, United Kingdom}
\email{\href{mailto:parth.shimpi@glasgow.ac.uk}{parth.shimpi@glasgow.ac.uk}}



\begin{document} 

\begin{abstract} 
    We establish, for each orbifold crepantly resolving a Kleinian singularity,
    the existence of the cohomological Hall algebra (COHA) of coherent sheaves
    supported on the exceptional locus and explicitly compute this COHA as a
    completion of some positive half of the associated affine Yangian. Tracking
    these categories under derived autoequivalences and the McKay
    correspondence, we show that (1) every point in Bridgeland's space of
    stability conditions on the resolution arises from a Kleinian orbifold, and
    (2) every positive half of the affine Yangian can be recovered from the COHA
    associated to some such stability condition. This provides the first example
    of a family of (pointwise) COHAs defined over the space of stability
    conditions.
\end{abstract}

\maketitle
\thispagestyle{empty}



\textit{Cohomological Hall algebras} (COHAs, for short) associated to
two-dimensional algebro-geometric categories are expected to geometrically
realize \textit{positive halves} of Yangians; this expectation indeed holds true
for the Abelian categories of 0-dimensional sheaves on smooth surfaces
\cite{MMSV} and of preprojective representations of quivers \cite{BD_Okounkov,
SV_Yangians=COHA}.

The present work showcases the `whole' quantum group as an invariant of the
\textit{derived category}, from which \textit{all} positive halves
can be recovered as the COHAs of \textit{hearts of bounded $t$-structures}. The
hearts we consider and their COHAs, and hence also the positive halves of
Yangians, are naturally parametrised over the space of Bridgeland stability
conditions.

The illustration is most lucid in the familiar setting of the McKay correspondence, i.e.\ for the minimal resolution $\pi\colon\rsv\to \ssv$ of a Kleinian surface singularity $\ssv$. The derived category of interest is the full subcategory $\catDb_C(\rsv)\subset \catDb(\rsv)$ containing complexes supported on the $\pi$-exceptional fiber, whose stability manifold has a distinguished connected component $\Stab^\circ(\rsv)$ identified in \cite{Stab_Kleinian}. The quantum group of interest, an \textit{affine Yangian}, arises from the affine ADE quiver $\qv=(I,\Omega)$ associated to $\ssv$.

\begin{theoremintroduction}[{(=\ref{thm:allstabs},~\ref{thm:coha-surface-as-limit2},~\ref{thm:coha-surface-as-limit3})}] \label{thm:genericfirst}
    Let $H\subset \catDb_C(\rsv)$ be the heart of a $t$-structure, arising as
    $\calP(0,1]$ or $\calP[0,1)$ for some stability condition $(Z,\calP)\in
    \Stab^\circ(\rsv)$. Writing $\bfH$ for the derived moduli stack of objects
    in $H$, the Borel--Moore homology $\HBMbullet(\bfH)$ canonically admits the
    structure of a $\N\times \Z I$-graded, topologically complete, cohomological
    Hall algebra $\coha_H$. 

    The construction in fact holds equivariantly with respect to any suitable
    torus $A$, and the COHA $\coha_H^A$ thus obtained is isomorphic (as a graded
    topological algebra) to an explicit limit of subquotients of the
    affine Yangian $\Y_{\qv;\ A}$. In the non-equivariant setting this recovers
    a completion of $\sfU(\frakn)$ for some positive half $\frakn$ of the
    elliptic Lie algebra $\frakgell$ associated to $\qv$. 

    Furthermore, every positive half of $\frakgell$ corresponds to some heart in
    this fashion.
\end{theoremintroduction}
In particular we obtain a geometric realization, at the level of affine Yangians, of each nonstandard half of $\frakgell$ computed by Kac and Jacobsen~\cite{Jakobsen-Kac-Borel, Jakobsen-Kac-Borel-II}. 

\medskip

The analysis of Hall algebras under derived equivalences goes back to Cramer~\cite{Cramer_Hall}, who showed that while two derived-equivalent hereditary Abelian categories $\scrA,\scrB$ may have non-isomorphic extended \textit{functional} (i.e.\ usual) Hall algebras $\bfH_\scrA\not\simeq\bfH_\scrB$, the derived equivalence does induce an identification $\mathbf{DH}_\scrA \simeq \mathbf{DH}_\scrB$ of their \textit{(reduced) Drinfeld doubles} with respect to the natural coproducts. Thus $\bfH_\scrB$ may be viewed as a nonstandard half of $\mathbf{DH}_\scrA$, and vice versa.

The analogue of Cramer's theorem for \textit{cohomological} Hall algebras is currently unavailable for several reasons --- first, there is no known general construction of a coproduct, except in the case of $2$-Calabi--Yau categories (see \cite{DHKSV}). Moreover even when this suitable ``double'' can be defined, its non-standard halves are not immediately realised as COHAs. The above result, in exhibiting that all non-standard halves \textit{can} in fact be naturally completed to COHAs, provides the first compelling evidence for the existence of an analogue of Cramer's theorem.

\subsection*{Hearts of $t$-structures}

Of course it is possible (and necessary) to explicitly classify all hearts $H$
that can arise in Theorem~\ref{thm:genericfirst}. The Kleinian singularity $\ssv$ can be modelled as a quotient
$\C^2/G$ for some finite subgroup $G\subset \SL(2,\C)$, and is hence crepantly
resolved by the Deligne--Mumford stack $\calX_\emptyset \coloneqq [\C^2/G]$. A
dense subset of $\Stab^\circ(\rsv)$ contains stability conditions on images of
$\catCoh(\calX_\emptyset)\subset \catDb(\calX_\emptyset)$ under various
equivalences $\catDb(\calX_\emptyset)\to \catDb(\rsv)$, and indeed, this
property distinguishes the component $\Stab^\circ(\rsv)$ in the first place. The
equivalences in question are the derived McKay
correspondence~\cite{Kleinian_derived} and compositions thereof with the
standard action of the extended affine braid group
$B_\sfex(\qv)\circlearrowright\catDb\rsv$ by spherical twists (see
\S\ref{subsec:reflection-functors} and \S\ref{subsec:geometricbraidgroup}).

There is in fact a family of \textit{Kleinian orbifolds} interpolating the two resolutions $X$ and $\calX_\emptyset$. Writing $\{C_i\;|\; i\in I_\sff\}$ for the set of irreducible $\pi$-exceptional curves in $X$, for each $J\subset I_\sff$ we may freely blow down the curves $\{C_i\;|\; i\notin J\}$ to obtain a \textit{partial resolution} $\psv$ of $\ssv$. This surface $\psv$, which again has Kleinian singularities, is the coarse moduli space of a smooth Deligne--Mumford stack $\calX_J$ which must therefore crepantly resolve $\psv$.

\begin{minipage}[c]{0.6\textwidth}
\begin{align}
	\begin{tikzcd}[ampersand replacement=\&]
        \rsv \arrow[rrd, bend left=30, "\pi"] \arrow[rd, "\pi_J"']
		\&\&  \\
		\&\psv \arrow[r, "\varpi_J"']\& \ssv \\ 
        \calX_J \arrow[ru]
	\end{tikzcd} 
\end{align}
\end{minipage}
\begin{minipage}[c]{0.3\textwidth}
        \includegraphics[width=\textwidth]{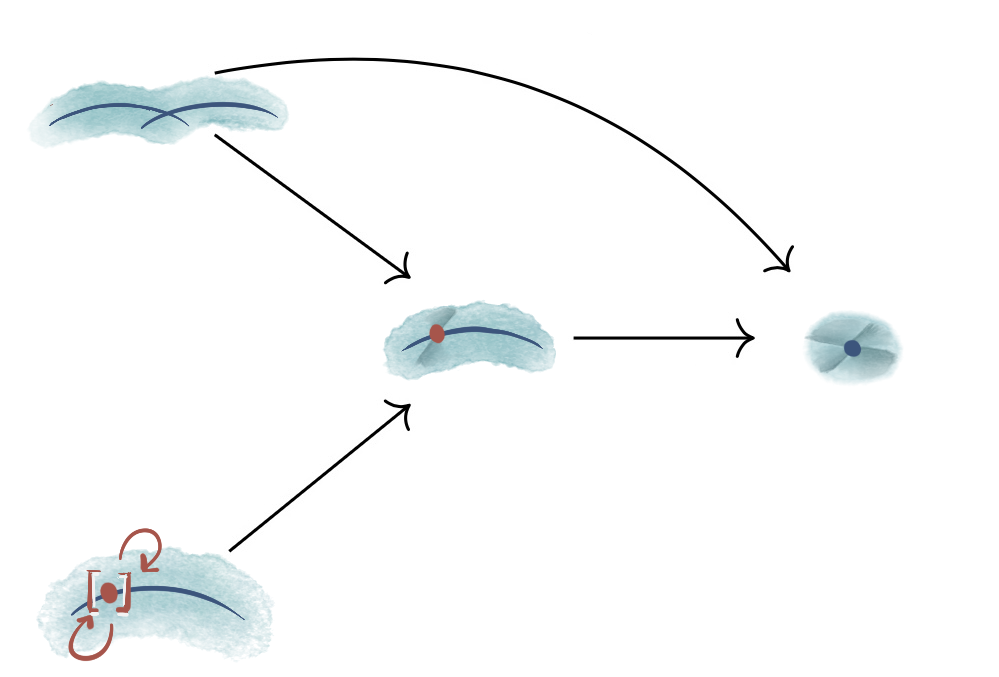}
\end{minipage}

The various orbifolds $\calX_J$ thus obtained, indexed over subsets $J\subset I_\sff$, are derived equivalent and we fix a choice of equivalences $\catDb\calX_J\to \catDb X$ (\S\ref{subsec:psheavesres}). Write $\catP(\rsv/\psv)$ for the image of $\catCoh(\calX_J)$ under this equivalence, and $\catPC(\rsv/\psv)$ for the induced heart $\catP(\rsv/\psv)\cap \catDbC (\rsv)$ in $\catDbC(\rsv)$. We also consider the category $\barcatCoh(\calX_J)$ obtained by tilting $\catCoh(\calX_J)$ in the torsion class of sheaves with $0$-dimensional support, and the $t$-structure $\barcatPC(\rsv/\psv)\subset \catDbC(X)$ it analogously induces.

\begin{theoremintroduction}[{(=\ref{thm:allstabs})}] \label{thm:A} 
    Given any stability condition $(Z,\calP)\in \Stab^\circ(\rsv)$, there exists a subset $J\subset I_\sff$ and an element $b\in B_\sfex(\qv)$ such that the heart $b\cdot \calP(0,1]$ is equal to $\catPC(\rsv/\psv)$ or a shift
    thereof. The heart $b\cdot \calP[0,1)$ in this case, up to said shift, is equal to $\barcatPC(\rsv/\psv)$.
\end{theoremintroduction}                                                           

This is not surprising, it is in fact expected that the list of $t$-structures
arising from Kleinian orbifolds is exhaustive up to modifications in point
sheaves and the action of $B_\sfex(\qv)$. The expectation, confirmed above for
$t$-structures that admit stability functions, has also been confirmed for
$t$-structures intermediate with respect to $\catPC(\rsv/\ssv)$ in
\cite{Shimpi_Torsion_pairs}.

\subsection*{Cohomological Hall algebras} 

Following~\cite{Porta_Sala_Hall, DPSSV-1} denote by $\dstackCoh_C(\calX_J)$ the derived moduli stack of coherent sheaves on $\calX_J$ (for some $J\subseteq I_\sff$) set-theoretically supported within the exceptional locus, and consider the \textit{convolution diagram}
\begin{align}
	\begin{tikzcd}[ampersand replacement=\&]
        \dstackCoh_C(\calX_J) \times \dstackCoh_C(\calX_J) \&
        \dstackCoh^{\mathsf{ext}}_C(\calX_J) \ar{r}{p} \ar{l}[swap]{q}\&
        \dstackCoh_C(\calX_J) 
	\end{tikzcd}\ .
\end{align}
Here $\dstackCoh^{\mathsf{ext}}_C(\calX_J)$ denotes the derived stack parametrising short exact sequences, and the maps $p,q$ map such a sequence to its middle (resp.\ extreme terms). 

We show this correspondence does in fact induce an associative product on the Borel--Moore homology of $\dstackCoh_C(\calX_J)$, which may be taken equivariantly with respect to any diagonal torus $A\subset \SL(2,\C)$
centralising $G$, and identify the resulting algebra with a quantum group explicitly constructed from the affine ADE quiver $\qv$ associated to $G$.

\begin{theoremintroduction}[{(=\ref{thm:coha-surface-as-limit2},~\ref{thm:coha-surface-as-limit3})}] \label{thm:B}
	For $J \subseteq I_\sff$, $A$, and the morphisms $p,q$ as above, the following hold.
	\begin{enumerate}\itemsep0.2cm
		\item The operation $p_\ast\circ q^!$ on $A$-equivariant Borel--Moore homology canonically endows the topological vector space 
		\begin{align}
			\coha^A_J\coloneqq \sfH_\bullet^A(\dstackCoh_C(\calX_J))
		\end{align}
         with the structure of an $\N \times \Z I$-graded, topologically complete, cohomological Hall algebra.
		
		\item There is an isomorphism of graded topological algebras 
		\begin{align}
            \Phi\colon \coha^A_J \xrightarrow{\quad\sim\quad} \Y^+_{J;A}\ ,
		\end{align} 
        where $\Y^+_{J;A}$ is an explicit limit of subquotients of the affine Yangian $\Y_{\qv;\ A}$ associated to $\qv$.

        For $A=\{\id\}$, $\coha_J$ is equal to the completed enveloping algebra $\widehat{\sfU}(\fraknellJ^+)$ of
		\begin{align}
			\fraknellJ^+\coloneqq \bigoplus_{\alpha \in \Delta^+_\sff \smallsetminus \Delta^+_{\Jcomp}}\hspace{-5pt}\frakg_{\alpha}[s^{\pm 1},t] \ \oplus\   \frakn_{\Jcomp}[t] \ +\  s^{-1}\frakh[s^{-1},t]\  \oplus \ K_-\ ,
		\end{align}
            where $\Delta_\sff$ is the root system associated to the (finite-type) ADE quiver $\qvfin=(I_\sff,\Omega_\sff)$ inside $\qv$, $\Delta_\Jcomp$ is the root subsystem spanned by the simple roots associated to $\Jcomp\coloneqq I_\sff\smallsetminus J$, the algebra $K_-$ is the negative half of the (infinite-dimensional) center of the elliptic Lie algebra $\frakgell$ associated to $\qv$, and $\frakn_{\Jcomp}$ is the standard negative half of the affinization of the $\Jcomp$-Levi subalgebra of the semisimple Lie algebra $\frakg$ associated to $\qvfin$.
	\end{enumerate}
\end{theoremintroduction}

An analogous result holds for the derived moduli stacks of objects in $\barcatCoh(\calX_J)$. Theorem~\ref{thm:genericfirst} then follows from Theorems~\ref{thm:A} and \ref{thm:B}, and the observation that the nonstandard positive halves $\fraknellJ^+$ which appear above precisely match those classified in the context of affine root systems by Kac and Jacobsen in~\cite{Jakobsen-Kac-Borel, Jakobsen-Kac-Borel-II}. 

\subsection*{Limiting COHAs and stability arcs} 

The existence of a Hall product boils down to the question of whether the
pushforward $p_\ast$ and Gysin pullback $q^!$ on Borel--Moore homology exist. A
sufficient condition is when $p$ is \textit{locally rpas}\footnote{In the sense
of
\cite[Definition~\ref*{torsion-pairs-def:modified_classes_of_morphisms}--(\ref*{torsion-pairs-item:locally-rpas})]{DPS_Torsion-pairs}.}
and $q$ is quasi-compact, \textit{finitely connected}\footnote{In the sense of
\cite[Definition~\ref*{torsion-pairs-def:modified_classes_of_morphisms}--(\ref*{torsion-pairs-item:finitely-connected})]{DPS_Torsion-pairs}.},
and derived lci. By \cite{Porta_Sala_Hall, DPS_Torsion-pairs} the latter
condition is guaranteed by the $2$-Calabi--Yau setting at hand, and properness
of Quot schemes is the pivotal question that we address via the theory of
\textit{limiting COHAs}, developed by the first- and second-named authors,
together with Diaconescu, Porta, and Vasserot, in
\cite[Part~\ref*{COHA-Yangian-part:COHA-stability-condition}]{DPSSV-3}.

To explain the idea, first employed in~\cite{DPSSV-1} to address $\catCohC(\rsv)$ (i.e.\ the case $J=I_\sff$), we note that COHA of $\catCohC(\calX_\emptyset)$ exists and has been studied extensively in
\cite{SV_generators, SV_Yangians, SV_Yangians=COHA,DPSSV-1}. Indeed, the derived stack $\dLambda_\qv\coloneqq \dstackCoh_C([\C^2/G])$ parametrises $G$-equivariant sheaves supported on the origin of $\C^2$, equivalently nilpotent preprojective representations of $\qv$, and hence is derived lci with quasi-compact connected components of finite type over $\C$.

The existence and computation of Hall product on $\sfH_\bullet^A(\dstackCoh_C(\calX_J))$ are then simultaneously deduced by \textit{approximating} the stack $\dstackCoh_C(\calX_J)$ by a sequence of locally closed substacks $\tensor*[^k]{\dLambda}{_\qv}\subset \dLambda_\qv$, identified along the action of certain braids $b^k\in B_\sfex(\qv)$. In \cite{DPSSV-3} this approximation is achieved in the case $J=I_\sff$ by explicitly identifying a Bridgeland stability condition $(Z,\calP)$ and a sequence of positive real numbers $t_k$ limiting to
$\half$ such that 
\begin{align}\label{eqn:introapproximation}
    \calP(-t_k,1-t_k]=b^k\,\catPC(\rsv/\ssv)\quad \text{and} \quad \calP(-\half,\half]=\catCohC(\rsv)\ . 
\end{align}

\begin{minipage}[c]{0.63\textwidth}
    The novelty in the present work stems from observing, following~\cite{Shimpi_Torsion_pairs}, that the existence of $(Z,P)$ as above is simply an artefact of $\catCohC(\rsv)$ being the infimum of a decreasing sequence
    $\catPC(\psv/\ssv)>b\,\catPC(\psv/\ssv)>b^2\,\catPC(\psv/\ssv)>\cdots$ in the lattice of $t$-structures intermediate with respect to $\catPC(\psv/\ssv)$. In \textit{loc.\ cit.\ }the third-named author furnishes a complete description of said lattice by constructing the \textit{heart fan}, a complete simplicial fan in Euclidean space (the figure alongside corresponds to the $\sfA_2$ singularity) whose cones are naturally associated to intermediate $t$-structures.
\end{minipage}\hspace{0.01\textwidth}
\begin{minipage}[c]{0.36\textwidth}
        \includegraphics[width=\textwidth]{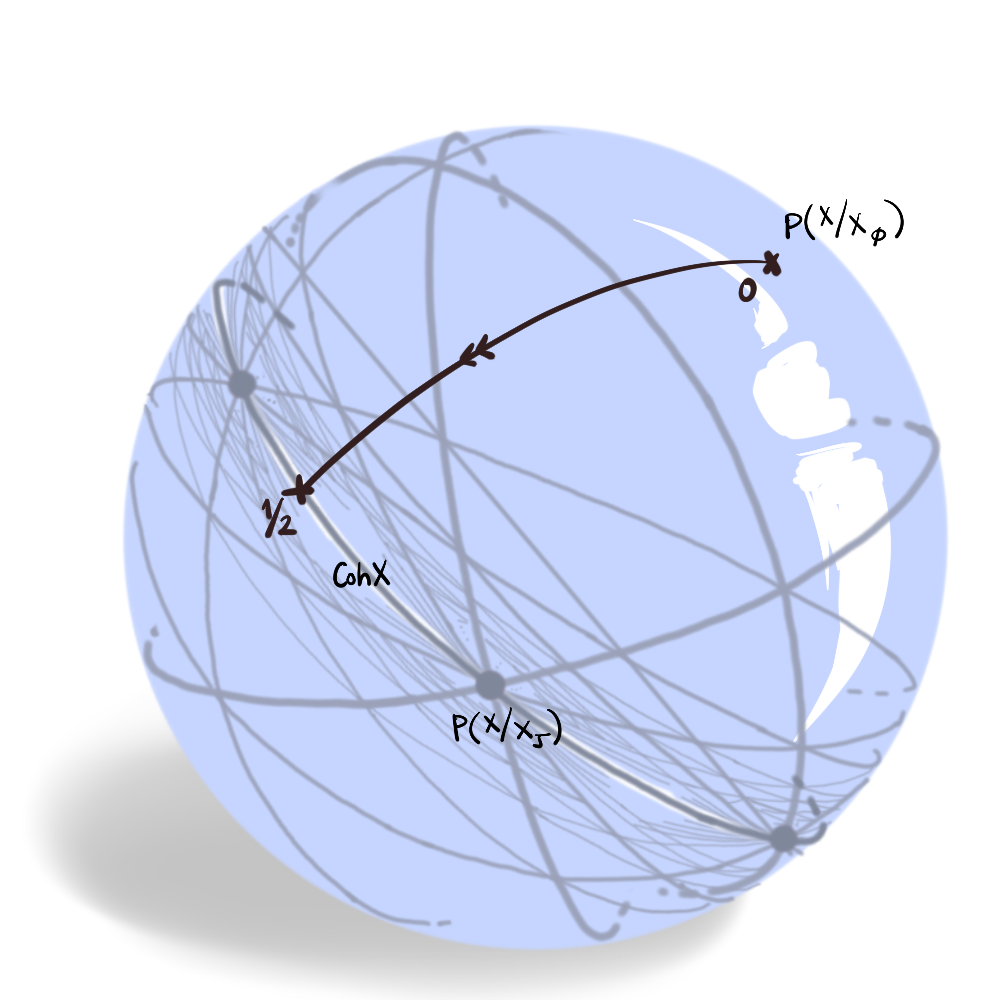}
\end{minipage}

Importantly the hearts $\catPC(\rsv/\psv)\simeq \catCohC(\calX_J)$ are all
intermediate with respect to $\catPC(\rsv/\ssv)$, and can be readily realized as
the infima of analogous sequences in the $B_\sfex(\qv)$-orbit of
$\catPC(\rsv/\ssv)$. The heart fan, being the universal phase diagram for
stability conditions \cite[\S6]{HeartFan}, also readily lets us read off precise
stability conditions $(Z,\calP)$ and real numbers $t_k$ such that the analogue
of the above relation for $\calP(-t_k,1-t_k]$ and $\calP(-\half,\half]$ holds
--- in \S\ref{subsec:stabilityarc}, which may be of independent interest, we
distil the procedure to simply the construction of certain \textit{stability
arcs}, paths $\gamma$ from the unit interval $[0,1]$ into the heart fan, which
can be lifted to paths $\overline\gamma:[0,1]\to \Stab^\circ(\rsv)$ such that
$\gamma(t)$ lies in the heart cone of the standard heart of
$\overline\gamma(t)$.

Thus our techniques recover not only the particular approximation found in \cite{DPSSV-3}, but also \textit{every other} approximation result of that shape.

\subsection*{Towards a sheaf of COHAs on the stability manifold}

It is natural to ask whether one can canonically construct, over
$\Stab^\circ(\rsv)$, a sheaf of associative algebras which fibre-wise recovers
the COHAs of Theorem~\ref{thm:genericfirst}. This turns out to be a highly
non-trivial problem, as we now explain.

Let $ \dstackPerfps(\rsv)$ denote the Toën--Vaquié derived moduli stack of pseudo-perfect complexes on $\rsv$, and let $ \dstackPerfps(\rsv)^{\mathsf{an}}$ be its analytification in the sense of \cite{Porta_Yu_Higher_Analytic_Stacks}. Consider the derived substack\footnote{We thank Mauro Porta for suggesting this definition.}
\begin{align}
	\dstackCoh_{/\Stab^\circ(\rsv)}\subset \Stab^\circ(\rsv)\times \dstackPerfps(\rsv)^{\mathsf{an}}
\end{align}
parametrising pairs $((Z,\calP),E)$, where $(Z,\calP)\in \Stab^\circ(\rsv)$ and $E$ is a pseudo-perfect complex on $\rsv$ that is flat\footnote{In the sense of~\cite[Definition~\ref*{torsion-pairs-def:flatness_wrt_t_structure}]{DPS_Torsion-pairs}.} with respect to the $t$-structure whose heart is $\calP(0,1]$.

Ideally, one would like to define a corresponding 2-Segal space $\calS_\bullet \dstackCoh_{/\Stab^\circ(\rsv)}$, relative to $\Stab^\circ(\rsv)$, together with a convolution diagram over $\Stab^\circ(\rsv)$,
\begin{align}
	\begin{tikzcd}[ampersand replacement=\&]
		\dstackCoh_{/\Stab^\circ(\rsv)} \times \dstackCoh_{/\Stab^\circ(\rsv)} \& \calS_2\dstackCoh_{/\Stab^\circ(\rsv)} \ar{r}{p} \ar{l}[swap]{q}\& \dstackCoh_{/\Stab^\circ(\rsv)}
	\end{tikzcd}\ .
\end{align}

Several difficulties arise in pursuing this program. First, it is not clear that $\dstackCoh_{/\Stab^\circ(\rsv)}$ is a derived Artin stack. Assuming this holds, it remains uncertain whether the sheaf $\pi_\ast \D\Q$ on
$\Stab^\circ(\rsv)$ --- where $\D\Q$ denotes the dualizing complex of
$\dstackCoh_{/\Stab^\circ(\rsv)}$ and $\pi\colon
\dstackCoh_{/\Stab^\circ(\rsv)}\to \Stab^\circ(\rsv)$ is the structure morphism
--- admits a natural algebra structure induced by $p_\ast \circ q^!$. Finally, even if such a structure exists, proving that the fibre of $\pi_\ast \D\Q$ at $(Z,\calP)\in \Stab^\circ(\rsv)$ coincides with the COHA associated with the heart $\calP(0,1]$ would be an extremely delicate task.

We will pursue this program and address these issues in the future.

\subsection*{Outline}

We begin, in \S\ref{sec:affine-quivers}, with a reminder on affine root systems and braid groups to fix our notations. In \S\ref{sec:perverse-coherent-sheaves} we introduce the Kleinian orbifolds, and the identification of the natural hearts therein with Van den Bergh's perverse coherent $t$-structures. This section also proves Theorem~\ref{thm:A}, by explicitly computing stability functions admitted by the orbifold hearts and the
orbits of these under braid group actions. The key result expressing all $t$-structures as limits of suitable translates (under the affine braid group) of $\catPC(\rsv/\ssv)$ is proved in \S\ref{sec:numerical-tilts}, see Theorem~\ref{thm:slicing}. This is then used in \S\ref{sec:limitingCOHAs} and \S\ref{sec:limit-affine-Yangian} to compute explicitly the associated COHAs $\coha^A_J$ as limiting COHAs, and describe their classical limits in terms of elliptic root systems in \S\ref{sec:classical-limit}.

\subsection*{Notation}

Unless otherwise specified, we work over the ground field $\C$ and all the stacks or algebraic varieties which we consider are defined over $\C$. 

For any variety $T$, we write $\catDb_{\ast}(T)$ for the bounded derived category $\catDb_{\ast}(\catCoh(T))$ of the Abelian category $\catCoh(T)$ of coherent sheaves on $T$, where $\ast$ indicates any property imposed on the cohomology objects of the bounded complexes. Similarly, for a sheaf $\calA$ of algebras over $T$, we denote by $\catDb_{\ast}(\calA)$ the bounded derived category $\catDb_{\ast}(\catCoh(\calA))$ of the Abelian category $\catCoh(\calA)$ of coherent sheaves on $T$ which are $\calA$-modules.

\subsection*{Funding}

The first-named author is partially supported by JSPS KAKENHI Grant number JP18\-K13402. This work was also partially supported by the ``National Group for Algebraic and Geometric Structures, and their Applications'' (GNSAGA-INDAM). Moreover, the first-named author acknowledges the MIUR Excellence Department Project awarded to the Department of Mathematics, University of Pisa, CUP I57G22000700001.

The second-named author is partially funded by the PNRR Grant `Cohomological Hall algebras and smooth surfaces and applications', CF 44/14.11.2022. 

The third-named author is supported by ERC Consolidator Grant 101001227 (MMiMMa).

\subsection*{Open access}

The authors have applied a Creative Commons Attribution (CC:BY) licence to any Author Accepted Manuscript version arising from this submission.

\subsection*{Acknowledgments}

Many thanks to Mauro Porta and Timothy De Deyn for enlightening conversations. The first-named author is a member of GNSAGA of INDAM.

\section{Affine quivers: Lie theory and representations}\label{sec:affine-quivers}

Fix a finite subgroup $G\subset \mathsf{SL}(2, \C)$ and let $\qvfin=(I_\sff, \Omega_\sff)$ be a Dynkin quiver of the corresponding ADE type, with vertex set $I_\sff=\{1, \ldots, e\}$ and an arbitrarily chosen edge orientation. The quiver $\qvfin$ admits an \textit{affine} extension $\qv=(I, \Omega)$, where $I\coloneqq I_\sff\cup \{0\}$.

\subsection{Spherical and affine root systems}\label{subsec:affine-quivers-roots-coweights}

The quiver $\qv$ determines an affine Kac--Moody algebra $\frakg$. Write $\{\alpha_0,\ldots,\alpha_e\}$ for its set of simple roots, then the root and coroot lattices of $\frakg$ are, respectively,
\begin{align}
	\rootlattice\coloneqq\bigoplus_{i\in I}\, \Z \alpha_i \quad \text{and} \quad \corootlattice\coloneqq\bigoplus_{i \in I}\, \Z\check{\alpha}_i\ .
\end{align}
The associated subsets of real and imaginary roots are respectively denoted $\rootsetre$ and $\rootsetim$, so that $\rootsetre\cup \rootsetim=\Delta\subset \rootlattice$ is an affine root system. In particular, elements of $\rootsetim$ are precisely those non-zero elements of $\rootlattice$ for which the Cartan pairing
\begin{align}\label{eq:cartan-pairing}
	\langle -, - \rangle\colon \corootlattice\times \rootlattice\longrightarrow \Z
\end{align}
(given by the generalised Cartan matrix associated to $\qv$) vanishes identically, and the kernel $\rootsetim\cup \{0\}$ is a sublattice of $\rootlattice$. It has a minimal positive generator, the \textit{primitive imaginary root}
\begin{align}\label{eq:r_i}
	\delta\coloneqq \sum_{i\in I}\,r_i \cdot \alpha_i\ .
\end{align}
The integers $(r_i)_{i\in I}$ are all positive and satisfy $r_0=1$, and depend only on the Dynkin type of $\qv$. Their sum $h\coloneqq \sum_{i\in I}r_i$ is the \textit{Coxeter number} of $\qv$, and the explicit values of $h$ and $(r_i)_{i\in I}$ can be read off from \cite[Table~Z]{kacInfiniteRootSystems}.

Dual to the root lattice, we also have the \textit{coweight lattice} $\coweightlattice\coloneqq \Hom(\rootlattice,\Z)$ of $\frakg$, with basis given by the \textit{fundamental coweights} $\{\Lomega_0,\ldots,\Lomega_e\}$. We denote the canonical pairing between roots and coweights by
\begin{align}\label{eq:canonical-pairing}
	(-, -)\colon \coweightlattice\times \rootlattice\longrightarrow \Z\ ,
\end{align}
this is such that $(\Lomega_i, \alpha_j)=1$ if $i=j$ and $0$ otherwise. A coweight $\Lomega$ is said to be (\textit{strictly}) \textit{dominant} if $\Lomega(\alpha_i)\geq 0$ (resp.\ $>0$) for all $i\in I$.

Analogously, the finite type quiver $\qvfin$ defines a semisimple Lie algebra $\frakgfin$ with simple roots $\{\alpha_1,\ldots,\alpha_e\}$, and we define the associated root lattice $\rootlatticefin$, coroot lattice $\corootlatticefin$, and coweight lattice $\coweightlatticefin$ with its associated subset of (strictly) dominant coweights. The associated root system $\rootsetfin\subset\rootlatticefin$ is finite, in particular there is a unique \textit{highest root}
\begin{align}
	\varphi \coloneqq \sum_{i\in I_\sff}\,r_i\cdot \alpha_i\ .
\end{align}

It is convenient to view $\rootlatticefin$ as split quotient of $\rootlattice= \Z\delta \oplus \rootlatticefin$. Dualising induces a natural inclusion of coweight lattices. Writing $\{\Llambda_1,\ldots,\Llambda_e\}\subset \coweightlatticefin$ for the fundamental coweights of $\frakgfin$, this map is explicitly given by
\begin{align}\label{eq:Lthetaconditions}
	\begin{tikzcd}[ampersand replacement=\&, row sep=tiny]
		\coweightlatticefin \coloneqq \displaystyle \bigoplus_{i\in I_\sff} \Z\Llambda_i \ar[hook]{r} \& \coweightlattice\coloneqq \displaystyle\bigoplus_{i\in I}\Z\Lomega_i\\
		\Llambda_i \ar[mapsto]{r} \&\Lomega_i - r_i \Lomega_0
	\end{tikzcd}\ .
\end{align}
Furthermore, the Cartan pairing $\corootlatticefin\times\rootlatticefin\to \Z$ is non-degenerate, so there is a natural inclusion $\corootlatticefin \hookrightarrow \coweightlatticefin$ allowing the coroot lattice to be identified with a subgroup of the coweight lattice.

\subsection{Spherical, affine, and extended braid groups}\label{subsec:braid-group-affine-quiver}

Let $W=W_{\mathsf{aff}}$ be the \textit{affine Weyl group} of Dynkin type defined by $G$, i.e.\ the Weyl group of the affine $\mathsf{ADE}$ quiver $\qv$. This is the subgroup of $\GL(\rootlattice)$ generated by the set of \textit{simple reflections} $\{s_0,\ldots,s_e\}$, which act on $\alpha\in \rootlattice$ via
\begin{align}
	s_i(\alpha) \coloneqq \alpha - \langle \Lalpha_i,\alpha \rangle \cdot \alpha_i\ .
\end{align}
The analogously defined \textit{spherical} Weyl group $W_\sff$ associated to $\qvfin$ coincides with the finite subgroup $\langle s_1,\ldots,s_e \rangle \subset W$, and we make this identification henceforth.

The finite-type coweight lattice $\coweightlatticefin$ (hence also the corresponding coroot lattice $\corootlatticefin$) acts on $\rootlattice$ via shear matrices -- the coweight $\Llambda\in \coweightlatticefin$ gives a linear map $\ell_{\Llambda}\colon \rootlattice\to\rootlattice$ given by
\begin{align}\label{eq:ell_Llambda}
	\ell_{\Llambda}(\alpha)\coloneqq\alpha+(\Llambda,\alpha)\,\delta\ .
\end{align}
If $\Llambda\in\corootlattice$, then $\ell_{\Llambda}$ in fact lies in $W\subset \GL(\rootlattice)$. Furthermore $W$ is generated by the subgroups $W_\sff$ and $\{\ell_{\Llambda}\,\vert\, \Llambda\in\corootlattice\}$, giving a decomposition $W=W_\sff\ltimes \corootlatticefin$ where the action $W_\sff\circlearrowright\corootlatticefin$ is the dual of $W_\sff\circlearrowright \rootlatticefin$.

Likewise we define the \textit{extended affine Weyl group} as the subgroup of $\GL(\rootlattice)$ generated by $W_\sff$ and $\{\ell_{\Llambda}\,\vert\, \Llambda\in\coweightlatticefin\}$, i.e.\ the semidirect product $W_\sfex \coloneqq W_\sff\ltimes \coweightlatticefin$. This naturally contains $W$ as a subgroup, we have the identity $W_\sfex  = \Gamma\ltimes W$ where $\Gamma$ is the group of outer automorphisms of the underlying diagram of $\qv$. Accordingly, this subgroup $\Gamma$ acts on $\rootlattice$ by permutating simple roots.

\begin{remark}
    Note that the definition of $\ell_{\Llambda}$ differs from that in \cite[Formula~(\ref*{COHA-Yangian-eq:t_lambda})]{DPSSV-3} by a sign. 
\end{remark}

Each element $w$ of the Weyl group has a well-defined \textit{length} $\ell(w)$ given by the minimal number of simple reflections it factors into, and the length function on $W$ can be extended to one on $W_\sfex=\Gamma\ltimes W$ via $\ell(\gamma\ltimes w)=\ell(w)$. The length functions respect the inclusions $W_\sff\subset W\subset W_\sfex$, i.e.\ the length functions on $W_\sff$ and $W$ agree with the restrictions of that on $W_\sfex$.

In particular, the finite group $W_\sff$ has a unique longest element which we denote $w_0$. Then conjugation by $w_0$ is an automorphism of $W_\sff$ which permutes simple reflections, thus defining an involution of the underlying Dynkin graph $\kappa\in \Gamma$ such that
\begin{align}\label{eq:dynkininvolution}
    w_0s_iw_0=w_{\kappa(i)}\quad\text{and}\quad w_0(\alpha_i)=-\alpha_{\kappa(i)}
\end{align}
for all $i\in I_\sff$ \cite[Lemma~1.2]{iyamaTitsConeIntersections}. We also define the \textit{extended Weyl element}
\begin{align}
	\widetilde{w}_0\coloneqq \kappa\ltimes w_0\, \in \, \Gamma\ltimes W_\sff \;\simeq\; W_\sfex\ .
\end{align}

The length function allows us to define the braid group associated to each Weyl group. For $\ast\in \{\sff, \mathsf{aff}, \sfex\}$, the corresponding \textit{braid group} $B_\ast$ is defined via the presentation
\begin{align}\label{eq:defBex}
    B_\ast \coloneqq \left\langle\; \{T_w\,\vert\, w\in W_\ast\}\, \vert \, T_vT_w=T_{vw}\quad \text{whenever }\,\ell(vw)=\ell(v) + \ell(w) \right\rangle\ .
\end{align}
The \textit{spherical} braid group $B_\sff$, associated to $W_\sff$, can be identified with the subgroup $\langle T_1,\ldots,T_e\rangle\subset B$. Each braid group naturally surjects onto its associated Weyl group via $T_w\mapsto w$, and these surjections are compatible with the inclusions $W_\sff\hookrightarrow W$ and $B_\sff\hookrightarrow B$.

We abbreviate $T_{\Llambda}\coloneqq T_{t_{\Llambda}}$ for each coweight $\Llambda\in \coweightlatticefin$ and $T_i\coloneqq T_{s_i}$ for each affine simple reflection. For each $\Llambda\in \coweightlatticefin$ we also define an element $L_{\Llambda}\in B_\sfex$ by first expressing $\Llambda$ as a difference of dominant vectors $\Llambda=\Llambda_1-\Llambda_2$ and declaring
\begin{align}\label{def:Llambda}
    L_{\Llambda} \coloneqq T_{\Llambda_1}T_{\Llambda_2}^{-1}\ ,
\end{align}
noting that the definition does not depend on the choice of the decomposition of $\Llambda$. Elements of the above form commute with each other, and $\{L_\Llambda\,\vert\, \Llambda\in \coweightlatticefin\}$ is a subgroup of $B_\sfex$ isomorphic to the finite-type coweight lattice.

Since the length functions on $W$ and $W_\sff$ can be considered restrictions of that on $W_\sfex$, the \textit{affine braid group} $B$ coincides with the subgroup $\langle T_0,\ldots,T_e\rangle \subset B_\sfex$, and there is a decomposition $B_\sfex = \Gamma \ltimes B$.

Likewise, the \textit{finite braid group} $B_\sff$ coincides with the subgroup of $\langle T_1,\ldots,T_e \rangle \subset B_\sfex$. We recall the following result.
\begin{proposition}[{\cite[\S3.3]{Mc}}]\label{prop:characterization-extended-affine-braid-group}
	The extended affine braid group $B_\sfex$ is generated by the subgroups $B_\sff$ and $\{L_{\Llambda}\,\vert\,\Llambda\in\coweightlatticefin\}$, subject to the relations
	\begin{align}
		\begin{cases}
			T_iL_{\Llambda}=L_{\Llambda} T_i  & \text{when $s_i(\Llambda)=\Llambda$}\ ,\\[4pt]
			L_{\Llambda}=T_iL_{\Llambda-\Lalpha_i} T_i & \text{when $s_i(\Llambda)=\Llambda-\Lalpha_i$}\ .
		\end{cases}
	\end{align}
\end{proposition}

\subsection{Actions on coweights}

We recall classical facts and constructions from Coxeter theory, concerning the natural action of the Weyl group $W\subset \GL(\rootlattice)$ on the vector space of coweights $\coweightlattice\otimes\R \simeq \Hom(\rootlattice,\R)$.

The affine root system $\Delta\subset \rootlattice$ cuts a hyperplane arrangement in $\coweightlattice\otimes\R$ that is preserved by the $W$-action, and in particular there is an induced action on the set of maximal chambers (i.e.\ cones which are not proper faces of other cones). It is standard (see e.g.\ \cite{Humphreys_Coxeter}) that this set of chambers is a union of three orbits, which can be represented by the cones
\begin{align}
	\calC^+
        &\coloneqq
        \big\{\theta\in \coweightlattice\otimes \R \;\vert\;
            (\theta, w_0\alpha_i) \geq 0 \text{ for all }i\in I\big\}\ , \\
	\calC^-
        &\coloneqq \big\{\theta\in \coweightlattice\otimes \R \;\vert\;
        (\theta, w_0\alpha_i) \leq 0 \text{ for all }i\in I\big\}\ , \\
    \calC^0
        &\coloneqq \big\{\theta\in \coweightlattice\otimes \R \;\vert\;
        (\theta,\alpha_i) \geq 0\text{ for all }i\in I_\sff\ , \;
        (\theta, \delta)=0\big\}\ .
\end{align}
The action of $W$ on the orbit of $\calC^\pm$ is faithful, while $\calC^0$ is stabilised by ${1}\ltimes\corootlattice_\sff\subset W_\sff\ltimes \corootlattice_\sff = W$ so that $W_\sff\subset W$ acts faithfully and transitively on the orbit $W\cdot \calC^0$. Thus we have a decomposition into distinct chambers
\begin{align}\label{eq:titsdecomp}
    \coweightlattice\otimes \R = \bigcup_{w\in W}w\calC^+\; \cup \bigcup_{w\in W_\sff}w\calC^0\; \cup \bigcup_{w\in W}w\calC^-\ .
\end{align}

Note that the faces of the cones $\calC^0$ (resp.\ $\calC^\pm$) are in bijection with subsets $J\subset I_\sff$ (resp.\ $J\subset I$), and we write
\begin{align}
    \calC^0_J\coloneqq\{\theta\in \calC^0\,\vert\, (\theta,\alpha_i) = 0 \text{ for all }i\in I_\sff\smallsetminus J\}\ .
\end{align}

\medskip

In what follows it is necessary to consider actions of subgroups of $W$ which stabilise the faces of $\calC^0$, namely for $J\subset I_\sff$ we consider the subgroup $W(J)\subset W$ which fixes the face $\calC^0_J$ pointwise. Evidently this contains the parabolic subgroup $\langle s_{i}\,\vert\, i\in I_\sff\smallsetminus J\rangle$ and the coroot lattice $1\ltimes \corootlatticefin$.

Say $J\subset I_\sff$ is \textit{connected} if the full subquiver of $\qvfin$ spanned by $J$ is so. In this case $J$ spans a Dynkin quiver of $\mathsf{ADE}$ type, and hence we can read off the positive integers $(r(J)_i\,\vert\, i\in J)$ from \cite[Table~Z]{kacInfiniteRootSystems} as we did for $J=I_\sff$. For each such $J$, we can thus define a root
\begin{align}
	\alpha_{J}\coloneqq\delta - \sum_{i\in J} r(J)_i \cdot \alpha_i \, \in \rootlattice\ .
\end{align}

\begin{lemma}
    \label{lem:fundamentaldomain-WJ}
    Given $J\subset I_\sff$, a fundamental domain for the action of $W(J)$ on the half-space
    $\{\delta>0\}$ in $\coweightlattice\otimes\R$ is given by
    \begin{align}
        \calD_J\coloneqq\left\{ \theta\in \coweightlattice\otimes
        \R\;\middle\vert\;
        \begin{array}{l}
            (\theta,\delta) > 0\quad\text{and}\quad
            (\theta, \alpha_i) \leq 0\text{ for each }i\in I_\sff\smallsetminus J\\
            (\theta, 2\delta-\alpha_{\Jcomp})\geq 0
            \text{ for each connected component }\Jcomp\subset I_\sff\smallsetminus J
        \end{array}\right\}\ .
    \end{align}

    \begin{proof}
        Given any $\theta\in\{\delta>0\}\subset \coweightlattice\otimes \R$, we
        show that there exists a $w\in W(J)$ with $w\theta\in \calD_J$. Suppose
        first, for simplicity, that $\Jcomp=I_\sff\smallsetminus J$ is connected.

        Let $W'\subset W$ be the subgroup generated by simple reflections
        $\{s_{i}\,\vert\,i\in \Jcomp\}$ and the simple coroots
        $\{1\ltimes \check\alpha_{i}\,\vert\, i\in \Jcomp\}$, and note that $W'$ lies in $W(J)$. We see that
		\begin{align}
			W' = \langle s_i\,\vert\,i\in \Jcomp \rangle \ltimes
            \corootlattice_{\Jcomp} \ ,
		\end{align}
        where $\corootlattice_{\Jcomp}\subset \corootlatticefin$ is the
        sublattice spanned by $\{\alpha_{i}\,\vert\, i\in \Jcomp\}$, so
        that $W'$ is isomorphic to the affine Weyl group determined by the
        Dynkin quiver spanned by $\Jcomp$.

        Thus $W'$ acts on an affine root lattice $\rootlattice'$ of this Dynkin
        type, say with simple roots $\{\beta_i\,\vert\,i\in \Jcomp\}$ and
        $\beta_0$, and primitive positive imaginary root $\delta'=\beta_0 +
        \sum_{i\in \Jcomp} r(\Jcomp)_i \cdot  \beta_i$.

        Given $\theta\in \coweightlattice\otimes \R$ such that $\theta(\delta)>0$, define a coweight $\theta'\in \Hom(\rootlattice',\R)$ by
		\begin{align}
			\theta'(\beta_0)=\theta(\alpha_{\Jcomp})\ \quad\text{and}\quad \theta'(\beta_i)=\theta(\alpha_{i}) \quad\text{for }i\in \Jcomp\ .
		\end{align}
        A straightforward computation shows that for $w\in W'$, we have
		\begin{align}
			w\cdot \theta'(\beta_0)=w\cdot \theta(\alpha_{\Jcomp})\ \quad\text{and}\quad w\cdot\theta'(\beta_i)=w\cdot \theta(\alpha_{i}) \quad\text{for }i\in \Jcomp\ .
		\end{align}
        In particular $\theta'(\delta')=\theta(\delta)>0$, so there exists a $w\in W'$ such that $w\theta'$ lies in the Weyl chamber\footnote{The particular Weyl chamber chosen here is the image of the dominant chamber in the longest word $w_0'\in \langle s_{i}\,\vert\, i\in \Jcomp\rangle$.}
		\begin{align}
			\{\vartheta\in \Hom(\rootlattice',\R)\,\vert\, \vartheta(\beta_i)\leq 0 \text{ for all }i\in \Jcomp\;\text{and}\; \vartheta(2\delta'-\beta_0)\geq 0\} \ .
		\end{align}
        The above conditions translate to $w\theta \in \calD_J$ as required.

        Note that for $w\in W'$ as above and for any $i\in I_\sff$, the coordinates $\theta(\alpha_i)$ and $w\theta(\alpha_i)$ can differ only if $i$ is connected to some vertex in $\Jcomp\subset\qvfin$. Thus in the event that $\Jcomp$ spans a disconnected subquiver of $\qvfin$, say with connected components given by $J_1,\ldots,J_n\subset I_\sff$, we can first run the above algorithm for the pair $(J_1,\theta)$ and find a $w_1\in W$ such that
		\begin{align}
			w_1\theta(2\delta-\alpha_{J_1})\geq 0\ \quad\text{and}\quad
            w_1\theta (\alpha_i)\leq 0 \quad \text{ for all } i\in J_1\ ,
		\end{align}
        and then run the algorithm again for the pair $(J_2,w_1\theta)$ to find a $w_2$, noting that $w_2w_1\theta$ continues to satisfy the above constraint.
    \end{proof}
\end{lemma}

\subsection{Nilpotent representations of preprojective algebras}

The Lie theoretic constructions above manifest themselves naturally in the theory of preprojective representations of $\qv$ and associated stability data, we now recall the constructions.

Consider the \textit{double quiver} $\doubleqv=(I, \Omega\cup \Omega^\ast)$, where
\begin{align}
	\begin{tikzcd}[ampersand replacement=\&]
		\Omega^\ast \coloneqq\Big\{i\ar{r}{e^\ast}\& j \;\Big\vert\; j\ar{r}{e}\& i\, \text{ is an edge in } \Omega\Big\}
	\end{tikzcd}\ .
\end{align}
The \textit{preprojective algebra} $\Pi$ associated to $\qv$ is defined as a quotient of the path algebra of $\doubleqv$ by the relations
\begin{align}
	\sum_{e\in\Omega}(e^\ast e -e e^\ast)=0\ .
\end{align}
Representations (i.e.\ right modules) of this algebra form an Abelian category $\ModPi$, this has a full subcategory $\modPi$ consisting of finite dimensional representations. Writing $e_i\in \Pi$ for the idempotent (i.e.\ lazy path) at the $i$th vertex of $\doubleqv$, each representation $M\in \modPi$ thus has underlying vector space $\bigoplus_{i\in I}e_iMe_i$ and can accordingly be associated a \textit{dimension vector} $\dim(M)\in \N^I$.

In particular since $\qv$ has no edge-loops, there is for any $i\in I$ a unique $\Pi$-module $S_i$ with dimension $\dim(S_i)_j=\delta_{ij}$, i.e.\ the vector space underlying $S_i$ is one-dimensional and supported on the $i$th vertex. Evidently, each such $S_i$ is simple.

We write $\nilpPi$ for the smallest extension-closed full subcategory of $\modPi$ containing the simple modules $S_0,\ldots,S_e$. Thus $M\in \modPi$ lies in $\nilpPi$ if and only if it admits a finite Jordan--H\"older composition series with each composition factor isomorphic to some $S_i$. Such modules are called \textit{nilpotent}, and the terminology is justified by the following evident fact.
\begin{proposition}
    Let $\calI$ be the two-sided ideal of $\Pi$ which is generated by all the arrows of $\doubleqv$. Then $M\in \modPi$ lies in $\nilpPi$ if and only if $M\cdot \calI^\ell = 0$ for some $\ell \geq 1$.
\end{proposition}
Nilpotent representations form Serre subcategories of $\ModPi$ and $\modPi$, in particular $\nilpPi$ is Abelian. The existence of Jordan--H\"older filtrations shows that the Grothendieck group $\sfK_0(\nilpPi)$ is freely generated (as a $\Z$-module) by the classes $\{[S_0],\ldots,[S_e]\}$, and we fix once and for all the identification
\begin{align}\label{eq:K0rootidentification}
	\begin{tikzcd}[ampersand replacement=\&, row sep=tiny]
		\sfK_0(\nilpPi) \ar{r}{\sim} \&\rootlattice\\
		{[S_i]} \ar[mapsto]{r}\& w_0(\alpha_{i})
    \end{tikzcd}\ .
\end{align}
Explicitly, the involution $\kappa\in \Gamma$ allows us to compute the classes as
\begin{align}\label{eq:simpleclasses}
    [S_0]\mapsto 2\delta-\alpha_0\quad \text{and} \quad [S_i]\mapsto -\alpha_{\kappa(i)} \quad \text{for } i\in I_\sff\ .
\end{align}

\subsection{Braid autoequivalences}\label{subsec:reflection-functors}

Following \cite{BIRS09,SY13} we recall how the action $W\circlearrowright \rootlattice$ can be categorified to an action of the braid group $B$ on the bounded derived category $\catDb(\nilpPi)$, noting the canonical identifications $\sfK_0(\catDb(\nilpPi))=\sfK_0(\nilpPi)\simeq \rootlattice$.

The generator $s_i\in W$ acts on $\rootlattice$ via a simple reflection, the corresponding lift $T_i\in B$ acts via a \textit{derived reflection functor}. We shall now describe its construction. Let $e_i$ be the primitive idempotent of $\Pi$ corresponding to the vertex $i\in I$. This defines a two-sided ideal $I_i$ of $\Pi$ by
\begin{align}
	I_i\coloneqq \Pi(1-e_i)\Pi\ .
\end{align}
As $\qv$ has no edge loops, $I_i$ is a codimension $1$ ideal and the quotient $\Pi/I_i$ is precisely the simple module $S_i$. Furthermore each $I_i$ is a \textit{classical tilting} $\Pi$-module (i.e.\ tilting $\Pi$-module of projective dimension $\leq 1$), and all other tilting modules can be obtained as products of the $I_i$'s as follows.
\begin{theorem}[{\cite[Theorems~2.20, 2.21, and 2.26]{SY13}, \cite[Theorem~6.2]{IR08}}]\label{thm:reflection-functors}
    Each element $w\in W$ can be assigned a unique ideal $I_w\subset \Pi$ of finite codimension, such that the following conditions are satisfied.
	\begin{enumerate}\itemsep0.2cm
        \item The neutral element $e\in W$ is assigned the trivial ideal $I_e\coloneqq\Pi$.
        \item The simple reflections $s_i$ are assigned the ideals $I_{s_i}\coloneqq I_i$.
        \item Corresponding to each factorisation $w=uv$ with $\ell(w)=\ell(u)+\ell(v)$, there are natural isomorphisms of $(\Pi,\Pi)$-bimodules
		\begin{align}
			I_w \, \simeq \, I_u\cdot I_v \, \simeq\, I_u \otimes_\Pi I_v \, \simeq\, I_u\otimes_\Pi^\LL I_v\ .
		\end{align}
	\end{enumerate}
    Each $I_w$ is a tilting $\Pi$--module of projective dimension at most one, and multiplication in $\Pi$ gives a canonical isomorphism of algebras $\End_{\Pi}(I_w)\simeq \Pi$. Moreover, the above correspondence is a bijection between $W$ and the set of isomorphism classes of classical tilting $\Pi$-modules.
\end{theorem}

As a consequence, for each $w\in W$ the corresponding classical tilting module determines mutually quasi-inverse autoequivalences of $\catDb(\ModPi)$ given by
\begin{align}
	\R T_w \coloneqq \R\Hom_{\Pi}(I_w,-)\quad \text{and} \quad \LL T_w\coloneqq (-)\otimes^\LL_{\Pi}I_w\ .
\end{align}
These functors restrict to autoequivalences of the categories $\catDb(\modPi)$ and $\catDb(\nilpPi)$, see \cite[Theorem~2.3 and Lemma~2.22]{SY13}. Further, Theorem~\ref{thm:reflection-functors} also gives relations between the functors, and we have the following result.

\begin{theorem}[{{\cite[Proposition~2.27]{SY13}\cite[Theorem~6.6-(2)]{IR08}}}]\label{thm:twist}
    Each factorisation $w=uv$ in $W$ with $\ell(w)=\ell(u)+\ell(v)$ gives rise to natural isomorphisms of functors
	\begin{align}
		\R T_w \simeq \R T_u \circ \R T_v \quad \text{and}\quad \LL T_w \simeq \LL T_v \circ \LL T_u\ .
	\end{align}
    Consequently, the assignment $T_w\mapsto \R T_w$ defines a (left) action of the affine braid group $B$ on $\catDb(\ModPi)$, which restricts to actions on $\catDb(\modPi)$ and $\catDb(\nilpPi)$.
\end{theorem}

\begin{remark}
    In the Grothendieck group $\sfK_0(\nilpPi)\simeq \rootlattice$, the complex $\R T_w(S_i)$ has class $w\cdot [S_i]$ for all $i\in I$ and $w\in W$.  Thus the action $B\circlearrowright \catDb(\nilpPi)$ naturally induces the action $W\circlearrowright \rootlattice$.
\end{remark}

It is straightforward to extend the above action to an action of $B_\sfex=\Gamma\ltimes B$, by considering the natural action of $\Gamma$ on $\qv$ by outer automorphisms. Thus any $\pi\in\Gamma$ gives rise to an algebra automorphism $\pi\colon \Pi\to \Pi$, which produces autoequivalences of the associated module categories and their derived categories.

\begin{corollary}\label{cor:twist}
	The actions of $B$ and $\Gamma$ on $\catDb(\ModPi)$ give rise to an action of the extended affine braid group $B_\sfex$ on $\catDb(\ModPi)$, which restricts to actions on the subcategories $\catDb(\modPi)$ and $\catDb (\nilpPi)$.
\end{corollary}

Extending the notation of Theorem~\ref{thm:twist}, given any $w\in W_\sfex$ we write $\R T_w$ and $\LL T_w$ for the autoequivalences associated to the elements $T_w,T_w^{-1}\in B_\sfex$ respectively. For the simple reflections $1\ltimes s_i$ for $i\in I$, it is convenient to abbreviate the associated functors as $\R T_i$, $\LL T_i$ respectively. Likewise for the elements $t_\Llambda\in W_\sfex$ associated to coweights $\Llambda\in \coweightlattice$, we abbreviate the associated functors as $\R T_\Llambda$ and $\LL T_\Llambda$.

Likewise given a coweight $\Llambda=\Llambda_1-\Llambda_2$ where $\Llambda_1,\Llambda_2$ are dominant, we introduce the notation
\begin{align}
	\R L_{\Llambda}\coloneqq\R T_{\Llambda_1}\circ \LL T_{\Llambda_2}
\end{align}
for the autoequivalence corresponding to the element $L_\Llambda=T_{\Llambda_1}T_{\Llambda_2}^{-1}\in B_\sfex$.

\subsection{Harder--Narasimhan strata}

Vectors in the coweight space $\coweightlattice\otimes \R$, viewed as additive functions $\sfK_0(\nilpPi)\to \R$ on the Grothendieck group, can be used to construct stability conditions on the category $\nilpPi$ (see e.g. \cite[\S2.3]{Ginzburg_Quiver}).
\begin{definition}\label{def:slopefunction}
	Given $\theta\in \coweightlattice\otimes \R$, the \textit{$\theta$--slope} of a non-zero module $M\in \nilpPi$ is defined as
	\begin{align}
		\mu_{\theta}(M) \coloneqq \frac{( \widetilde{w}_0(\theta)\,,\,[M])}{\text{dim}(M)}\ ,
	\end{align}
    where $\text{dim}(M)$ is the dimension of the vector space underlying $M$ (equivalently, the number of simples appearing in its composition
    series).
\end{definition}
Explicitly, if $S_i$, for $i\in I$, appears $d_i$ times in a Jordan--Holder filtration of $M$ and $\theta=\sum_i \theta_i\Lomega_i$, then the slope is computed as
\begin{align}
	\mu_\theta(M)=\frac{\theta_0d_0+\cdots+\theta_ed_e}{d_0+\cdots+d_e}\ .
\end{align}

\begin{definition}
    We say the module $M$ is \textit{$\theta$--semistable} if the inequality $\mu_{\theta}(M')\leq \mu_{\Ltheta}(M)$ holds for every non-zero submodule $M'\subseteq M$. If the inequality is strict for every proper submodule, then we say $M$ is \textit{$\theta$--stable}.
\end{definition}

Given $\theta\in \coweightlattice\otimes\R$, each finite--dimensional nilpotent $\Pi$-module $M$ has a unique filtration
\begin{align}\label{eq:HN-filtration-representations}
    0 = M_{s+1}\subset M_s\subset \cdots\subset M_1=M\ ,
\end{align}
called its \textit{Harder-Narasimhan (HN) filtration}, such that the composition factors $M_i/M_{i+1}$ (called \textit{HN} factors) are all $\theta$--semistable, and the slopes $\mu_i \coloneqq \mu_{\theta}(M_i/M_{i+1})$ satisfy $\mu_s>\cdots>\mu_2>\mu_1$.

\begin{definition}\label{def:kappa}
	Given $\theta\in \coweightlattice\otimes\R$ and an interval $\kappa\subset\Q$, the \textit{Harder--Narasimhan stratum} $\nilp^{\kappa}_{\theta}\, \Pi$ is the full subcategory of $\nilpPi$ containing modules whose Harder--Narasimhan factors all have $\theta$--slopes in $\kappa$.
\end{definition}

We suppress $\theta$ from notation when its choice is clear, and use obvious shorthands (such as $\nilp^{\geqslant 0}\, \Pi$ for the stratum corresponding to $\kappa=[0,\infty)$) when convenient.

\section{Kleinian orbifolds and the stability manifold}\label{sec:perverse-coherent-sheaves}

The finite group $G\subset \SL(2,\C)$ acts on the affine space $\A^2$, and the quotient $\ssv\coloneqq \A^2/G$ has an isolated singularity at the origin.

For the remainder of the paper we fix a diagonal torus $A\subset \GL(2,\C)$ centralising $G$, so that $\ssv$ is equipped with an $A$-action. For $G$ of type $\sfA$, the diagonal torus $A$ could be $\{1\}$, $\G_m$, or $\G_m\times\G_m$, while for $G$ of types $\sfD$ or $\sfE$, it could be $\{1\}$ or $\G_m$.

\subsection{Kleinian orbifolds}\label{subsec:kleinian-orbifolds}

The Kleinian singularity $\ssv$ has an $A$-equivariant crepant resolution $\pi\colon \rsv\to \ssv$, which may be constructed as a Nakajima quiver variety parametrising $\delta$-dimensional stable representations of the affine preprojective algebra $\Pi$ (determined by $G$ as in \S\ref{sec:affine-quivers}).

The resolution has a connected exceptional fibre $C$ over the singular point, with the underlying reduced subscheme $C_\mathsf{red}$ isomorphic to a union of $\PP^1$s with nodal intersections. By the celebrated McKay correspondence \cite{GSV_McKay} these irreducible components $C_i\subset C_\text{red}$ can be naturally indexed over the vertices $i\in I_\sff$ of the finite-type Dynkin quiver $\qvfin$ associated to $G$, in a way that distinct vertices $i,j\in I_\sff$ are connected by an arrow in $\qvfin$ if and only if $C_i\cap C_j\neq \emptyset$.

There is in fact a family of Deligne--Mumford stacks $\calX_J$ (indexed over subsets $J\subseteq I_\sff$) which crepantly and $A$-equivariantly resolve the singular surface $\ssv$. These \textit{non-commutative resolutions}, which we call \textit{Kleinian orbifolds}, are constructed as follows -- the choice of $J\subseteq I_\sff$ determines an $A$-equivariant factorisation of $\pi$ given by
\begin{align}
	\begin{tikzcd}[ampersand replacement=\&]
	    \rsv \arrow[rr, "\pi"] \arrow[rd, "\pi_J"']
	    \&\& \ssv \\
	    \&\psv \arrow[ru, "\varpi_J"']\&
	\end{tikzcd}
\end{align}
where $\pi_J$ blows down the curves $\{C_i\,\vert\, i\in \Jcomp\}$, i.e.\ $\pi_J$ contracts each $C_i$ for $i\in \Jcomp\coloneqq I_\sff\smallsetminus J$ to a point and is an isomorphism away from this locus. The surface $\psv$ thus obtained has isolated Kleinian singularities. Consequently, there exists a \textit{canonical} stack $\calX_J$ associated to it --- this is a smooth Deligne--Mumford stack whose coarse moduli space is $\psv$, constructed as in \cite[Note~2.9 and proof of Proposition~2.8]{vistoliIntersectionTheoryAlgebraic}. The composite $\calX_J\to \psv\to \ssv$ is a crepant resolution.

The orbifold $\calX_J$, which adds stacky structure to the singularities of $\psv$, can equivalently realised by equipping $\psv$ with a non-commutative structure sheaf $\mathcal{A}_J$. The quiver variety $\rsv$ has a tautological vector bundle $\calP=\bigoplus_{i\in I}\calP_i$ that is globally generated and tilting, with $\calP_0=\scrO_\rsv$ and each $\calP_i$ indecomposable of rank $\delta_i$. Considering the sub-bundle 
\begin{align}
	\calP_J\coloneqq\bigoplus_{i\in I\smallsetminus J}\calP_i\ ,
\end{align}
Bertsch \cite[Theorem~1.4]{bertschSemiClassicalNoncommutativeResolutions} shows that the sheaf of $\calO_{\psv}$-algebras
\begin{align}
	\calA_J\coloneqq\pi_{J\ast}\calEnd_\rsv(\calP_J)
\end{align}
is Morita equivalent to $\calX_J$ i.e.\ $\catCoh(\calA_J)\simeq \catCoh(\calX_J)$ as (stacks of) Abelian categories. In particular the choice $J=\emptyset$ recovers the special case $\pi_\emptyset=\pi$, where $\calX_\emptyset=[\A^2/G]$ and $\Gamma(\ssv,\mathcal{A}_\emptyset)=\Pi$, and the Morita equivalences $\catCoh([\A^2/G]) \simeq \catmod(\C[\A^2]\ast G) \simeq \catmod(\Pi)$ are well known \cite{reitenTwodimensionalTameMaximal}.

\subsection{Derived equivalences}\label{subsec:psheavesres}

All crepant resolutions of $\psv$ are known to have equivalent derived categories, and specific equivalences between $\catDb(\rsv)$ and $\catDb(\calX_J)\simeq \catDb(\calA_J)$ can be chosen by noting that the morphism $\pi_J$ and the bundle $\calP_J$ satisfy the hypotheses of \cite[Propositions~3.2.7 and 3.3.1]{VdB_Flops}.

\begin{theorem}[{\cite[Proposition~3.3.1]{VdB_Flops}}]\label{thm:VdBmainthm}
    There are $A$--equivariant quasi-inverse equivalences
    \begin{align}
        \begin{tikzcd}[ampersand replacement=\&, column sep=10em]
            \catDb(\rsv)
            \arrow[r, "{\R\pi_{J\ast}\R\calHom(\calP_J,-)}", shift left=1]
            \& \catDb(\calA_J)
            \arrow[l, "{\pi_J^{-1}(-)
            \otimes^{\LL}_{\pi_J^{-1}(\calA_J)}\calP_J}", shift left=1]
        \end{tikzcd}\label{vdbequiv}
	\end{align}
    under which the heart of the standard $t$-structure $\catCoh(\calA_J)$ is identified with the heart of a \textit{perverse $t$-structure} $\catP(\rsv/\psv)\subset\catDb(\rsv)$, defined as the positive tilt of $\catCoh(\rsv)$ in the torsion pair
    \begin{align}\label{eq:pervtorsion}
    	\begin{split}
        \scrT_J&\coloneqq \left\{\calF\in \catCoh(\rsv)\, \middle\vert\,
            \begin{matrix}
                \R^1\pi_{J\ast}(\calF)=0\,\text{ and } \Hom(\calF, \calG)=0\\[2pt]
                 \text{for all $\calG\in \catCoh(\rsv)$ with }\R\pi_{J\ast}(\calG)=0
            \end{matrix}\right\}\ ,\\[2pt]
        \scrF_J&\coloneqq \left\{\calF\in \catCoh(\rsv)\;\middle\vert\, \R^0\pi_{J\ast} (\calF)=0\right\}\ .
        \end{split}
    \end{align}
\end{theorem}

By positive tilt, we mean that $\catP(\rsv/\psv)$ is smallest the extension--closed subcategory of $\catDb(\rsv)$ containing the full subcategory $\scrF_J[1]\cup\scrT_J$.

In the special case $J=I_\sff$, we have $(\psv,\calA_J)=(\rsv,\scrO_\rsv)$ and all functors are the identity, in particular $\catPC(\rsv/\rsv)$ equal to $\catCohC(\rsv)$. On the other hand when $J=\emptyset$ we get the equivalences
\begin{align}\label{eq:tau}
   \begin{tikzcd}[ampersand replacement=\&, column sep=10em]
       \catDb(\rsv)
       \arrow[r, "{\tau(-)\coloneqq\R\Hom(\calP_\emptyset,-)}", shift left=1]
       \& \catDb(\ModPi)
       \arrow[l, "{\tau^{-1}(-)\coloneqq(-) \otimes^{\LL}_{\Pi}\calP_\emptyset}", shift left=1]
   \end{tikzcd}\ ,
\end{align}
which when composed with the Morita equivalence $\ModPi\simeq \catMod(\C[\A^2]\ast G)$ recovers the derived McKay correspondence \cite{Kleinian_derived}.

\begin{proposition}[{\cite[Proposition~3.5.7]{VdB_Flops}}]\label{rem:simple}
    The simple objects of $\catP(\rsv/\ssv)$, i.e.\ the images of simple $\Pi$-modules $S_0,\ldots,S_e$ under the equivalence $\tau^{-1}$, are given by
	\begin{align}\label{eq:sphericalobj}
        \scrO_{C}\simeq\tau^{-1}(S_0)\quad \text{and} \quad
        \scrO_{C_i}(-1)[1]\simeq \tau^{-1}(S_{\kappa(i)})
        \quad\text{for } i\in I_\sff
	\end{align}
    where the involution $\kappa\colon I_\sff\to I_\sff$ is as in Formula~\eqref{eq:dynkininvolution}.
\end{proposition}

In what follows, we also consider for each $J\subset I_\sff$ the full subcategory $\catDbC(\calA_J)$ containing complexes set-theoretically supported on $\pi_J(C)\subset \psv$, i.e.\ the on the $\varpi_J$-exceptional fibre. This has a natural $t$-structure with heart $\catCohC(\calA_J)\coloneqq\catCoh(\calA_J)\cap \catDbC(\calA_J)$. Thus when $J=I_\sff$ we recover the Serre subcategory $\catCohC(\rsv)$ of coherent sheaves supported on $C$, while $J=\emptyset$ recovers the natural heart $\nilpPi\subset \catDb_\mathsf{nilp}(\ModPi)$.

By $\ssv$-linearity, the following is a consequence of Theorem~\ref{thm:VdBmainthm}.
\begin{corollary}
	The equivalences \eqref{vdbequiv} restrict to equivalences $\catDbC(\rsv)\simeq\catDbC(\calA_J)$, identifying $\catCohC(\calA_J)$ with $\catPC(\rsv/\psv)\coloneqq\catP(\rsv/\ssv)\cap \catDbC(\rsv)$.
\end{corollary}

 In particular there are natural identifications of Grothendieck groups
\begin{align}
	\sfK_0(\catCohC(\rsv))\;\simeq\; \sfK_0(\catDbC(\rsv)) \;\simeq\; \sfK_0(\catDb_\mathsf{nilp}(\ModPi))\;\simeq\; \sfK_0(\nilpPi)\ .
\end{align}
Recalling the identification $\sfK_0(\nilpPi)\simeq \rootlattice$ given in Formula~\eqref{eq:K0rootidentification}, from Proposition~\ref{rem:simple} we see that the sheaf
$\scrO_{C_i}(-1)\in \catCohC(\rsv)$ has $\sfK_0$-class $\alpha_i\in \rootlattice$, while $\scrO_C$ has class $2\delta-\alpha_0$. It follows that for any closed point $p\in C$, the skyscraper sheaf $\scrO_p$ has $\sfK_0$-class $\delta$.

\begin{remark}
    In the two cases $J=I_\sff$ and $J=\emptyset$ the two hearts $\catCohC(\calA_J)\subset \catDbC(\calA_J)$ have been computed to be \textit{faithful}, that is to say the induced realisation functor $\catDb(\catCohC(\calA_J))\to \catDbC(\calA_J)$ is an equivalence. For $J=I_\sff$ the equivalence $\catDb(\catCohC (\rsv))\simeq \catDbC(\rsv)$ is proved in \cite[Corollary~3.4]{Sheaves_local_CY} (see also \cite[Lemma~2.1]{orlovFormalCompletionsIdempotent}), while the $J=\emptyset$ case $\catDb(\nilpPi)\simeq \catDb_\mathsf{nilp}(\ModPi)$ is \cite[Theorem~1.4]{Lewis_stability}.
\end{remark}

\subsection{Stability functions on perverse hearts}\label{subsec:stabconds-perversehearts}

For each $J\subset I_\sff$ the category $\catPC(\rsv/\psv)$ is the heart of a $t$-structure in $\catDbC(\rsv)$, being the restriction of the heart $\catP(\rsv/\psv)\subset \catDb(\rsv)$. Indeed, $\catPC(\rsv/\psv)$ and can be obtained by tilting the natural heart $\catCohC(\rsv)\subset \catDbC(\rsv)$ in the induced torsion pair $(\scrT_J\cap \catCohC(\rsv), \scrF_J\cap\catCohC(\rsv))$.

In this subsection we characterise all the Bridgeland stability functions on such a heart, i.e.\ linear maps 
\begin{align}
	Z\colon \sfK_0(\catDbC(\rsv))\longrightarrow \C
\end{align}
such that $Z([x])\in \HH_-\coloneqq\{re^{i\pi\varphi}\,\vert\, r>0,\; 1\geq \varphi>0\}$ for all $x\in \catPC(\rsv/\psv)$, and such that local-finiteness and the Harder--Narasimhan properties\footnote{See \cite[Definitions~5.7 and 2.3]{B07} for the definitions of these notions.} are satisfied.

For Artinian and Noetherian hearts such as $\catPC(\rsv/\ssv)$ the two conditions are automatically satisfied as long as $Z$ maps every object of the heart into $\HH_-$, and further it suffices to check the latter condition on just the simple objects.

For the other hearts $\catPC(\rsv/\psv)$ we prove the following result, noting that any $Z\colon \sfK_0(\catDbC(\rsv))\to \C$ can naturally be seen as a complexified coweight (i.e.\ an element of $\coweightlattice\otimes\C$) across the identification $\sfK_0(\catDbC(\rsv))\simeq \rootlattice$.
\begin{theorem}\label{thm:perverse-stabfuncs}
    Given a non-empty subset $J\subset I_\sff$ and real coweights $\Lomega,\Ltheta\in \coweightlattice\otimes\R$, the complexified coweight $Z\coloneqq -\Ltheta+i\Lomega$ gives a stability function on $\catPC(\rsv/\psv)$
    if and only if $\Lomega\in \calC^0_J$ and $\Ltheta\in \calD_J$. All such stability functions are locally finite and satisfy the Harder--Narasimhan property.
\end{theorem}

To prove the above theorem, we need some preliminary results. We first decompose the `semigeometric' category $\catPC(\rsv/\psv)$ into simpler, purely algebraic and purely geometric pieces. Write $C_\Jcomp$ for the (not necessarily connected) union of the scheme-theoretic exceptional fibres of $\pi_J$, i.e.\
\begin{align}
	C_\Jcomp \coloneqq\bigcup_{\genfrac{}{}{0pt}{}{p\in \psv}{\text{singular}}}\; \pi_J^{-1}(p)\ .
\end{align}
This has underlying reduced subscheme $\bigcup_{i\in \Jcomp}C_i$. If $J_1,\ldots,J_n\subset I_\sff$ are the connected components of $\Jcomp$ (i.e.\ each $J_i$ determines connected component of the full subquiver of $\qvfin$ spanned by $I_\sff\smallsetminus J$), then we note that $C_\Jcomp=C_{J_1}\cup \ldots \cup C_{J_n}$.

Given such a curve $C_J\subset \rsv$, write $\catDb_{C_J}(\rsv)$ for the full subcategory of complexes in $\catDb(\rsv)$ set-theoretically supported within $C_J$, and define the subcategories
\begin{align}
	\catP_{C_J}(\rsv/\ssv)\coloneqq\catP(\rsv/\ssv)\cap  \catDb_{C_J}(\rsv)\quad\text{and}\quad\catCoh_{C_J}(\rsv)\coloneqq\catCoh(\rsv)\cap\catDb_{C_J}(\rsv).
\end{align}
Then we have the following result which shows that all the $\catPC(\rsv/\psv)$s are `built out of' the cases $J=\emptyset$ and $J=I_\sff$.

\begin{theorem}[{\cite[Theorem~5.23]{Shimpi_Torsion_pairs}}]\label{thm:structureofper}
    Given $J\subseteq I_\sff$, let $J_1,J_2,\ldots,J_n$ be the connected components of $I_\sff\smallsetminus J$. The heart $\catPC(\rsv/\psv)$ is the smallest extension-closed subcategory of $\catDbC(\rsv)$ containing the subcategories $\catP_{C_{J_1}}(\rsv/\ssv)$, $\catP_{C_{J_2}}(\rsv/\ssv)$, $\ldots$, $\catP_{C_{J_n}}(\rsv/\ssv)$, and $\catCoh_{C_J}(\rsv)$.
\end{theorem}

For $J\subset I_\sff$ and any conected component $J'\subset \Jcomp$, consider the Kleinian singularity $X'_\emptyset$ determined by the Dynkin subquiver of $\qvfin$ spanned by $J'$, and take its minimal resolution $X'\to X'_\emptyset$ with exceptional fibre $C'$. The singularity of $X'_\emptyset$ is étale-locally isomorphic to that at $\pi_J(C_{J'})\subset \psv$, so the categories $\catP_{C'}(\rsv'/X'_\emptyset)\subset \catDb(\rsv')$ and $\catP_{C_{J'}}(\rsv/\ssv)$ are naturally equivalent (see e.g.\ \cite[Corollary~2.9]{orlovFormalCompletionsIdempotent}). In particular, the latter is an Artinian and Noetherian Abelian category given as the extension closure of its simple objects
\begin{align}
	\catP_{C_{J'}}(\rsv/\ssv)=\left\langle \scrO_{C_{J'}}\,,\;
    \scrO_{C_i}(-1)[1] \;\middle\vert\;  i\in J' \right\rangle\ .
\end{align}

With this, we can prove the main result of this subsection.
\begin{proof}[Proof of Theorem~\ref{thm:perverse-stabfuncs}]
    Suppose $J\subset I_\sff$ is non-empty, and $Z=-\Ltheta+i\Lomega$ is a stability function on $\catPC(\rsv/\psv)$. Then for $i\in J$, we see that each of the sheaves $\{\scrO_{C_i}(n)\,\vert\, n\in \Z\}$ lies in $\catCoh_{C_J}(\rsv)$, and hence in $\catPC(\rsv/\psv)$. Examining the $\sfK_0$-classes $[\scrO_{C_i}(n)]=(n+1)\delta+\alpha_i$, we see that the condition $Z([\scrO_{C_i}(n)])\in \HH_-$ for all $n\in \Z$, $i\in J$ implies
	\begin{align}
		\Lomega(\delta)=0\quad\text{and}\quad \Lomega(\alpha_i)>0\quad\text{for all }i\in J\ .
	\end{align}
    Likewise for each connected component $J'\subset \Jcomp$, the stability
    function $Z$ must map the simple objects of
    $\catP_{C_{J'}}(\rsv/\emptyset)\subset \catPC(\rsv/\psv)$ into $\HH_-\subset
    \C$. For any $i\in J'$ the simple $\scrO_{C_i}(-1)[1]$ has class
    $-\alpha_i$, while the class of $\scrO_{C_{J'}}$ can be computed as
    \begin{align}
        [\scrO_{C_{J'}}] = [\scrO_p] - \sum_{i\in J'}r(J')_i\cdot
        [\scrO_{C_i}(-1)[1]] = 2\delta-\alpha_{J'}\ .
    \end{align}
    The constraints $\Lomega(-\alpha_i)\geq 0$ (for all $i\in J'$) and
    $\Lomega(2\delta-\alpha_{J'})\geq 0$ can be simultaneously satisfied if and only if
	\begin{align}
		\Lomega(\alpha_i)=0 \quad \text{for all } i\in J'\ .
	\end{align}
    It follows that $\Lomega\in \calC_J^0$. Likewise the constraints $\Ltheta(-\alpha_i)\geq 0$ and $\Ltheta(2\delta-\alpha_{\Jcomp})\geq 0$ translate to $\Ltheta\in \calD_J$ as required.

    Conversely, the above calculations show that any $Z$ of the given form maps all objects of $\catPC(\rsv/\psv)$ into $\HH_-\subset\C$, so it suffices to prove the local-finiteness and the Harder--Narasimhan properties. For the latter, by \cite[Proposition~2.4]{B07} it suffices to show that $\catPC(\rsv/\ssv)$ does not admit infinite sequences
	\begin{align}
		\ldots\subset x_{-2}\subset x_{-1}\subset x_0 \qquad \text{or}\qquad x_0\twoheadrightarrow x_1\twoheadrightarrow x_2\twoheadrightarrow \ldots
	\end{align}
    with $\varphi(x_{i})>\varphi(x_{i+1})$ for all $i$, where $\varphi\colon \catPC(\rsv/\psv)\to (0,1]$ denotes the phase function induced by $Z$ (so that for $x\in \catPC(\rsv/\psv)$ we have $Z([x])\in \mathbb{R}_{>0}\cdot e^{i\pi\varphi(x)}$.) We can preclude chains of the second kind by noting that $\catPC(\rsv/\psv)$ is the extension-closure of two Noetherian categories (Theorem~\ref{thm:structureofper}), hence is itself Noetherian. To address chains of the first kind, suppose $\ldots\subset x_{-2}\subset x_{-1}\subset x_0$ were such a chain, then each $x_i$ is a coherent sheaf generically on the curve $C_j$ for $j\in J$ so the generic rank of $x_i$ on such a $C_j$ is eventually constant. In particular the cokernel $x_i/x_{i-1}$ for $i\ll 0$ is supported on $C_{I_\sff\smallsetminus J}\cup\{p_1,\ldots,p_n\}$ for some closed points $p_1,\ldots,p_n\in C_J$. It follows that $x_i/x_{i-1}$ lies in $\langle \catP_{C_{I_\sff\smallsetminus J}}(\rsv/\ssv), \scrO_p\,\vert\, p\in C_J\rangle$, so has phase $1$. It follows that $\varphi(x_i)\geq \varphi(x_{i-1})$ for all $i\ll 0$. Thus all coweights in the subset $V=-\calD_J+i\calC_J^0\subset \coweightlattice\otimes \C$ parametrise stability functions on $\catPC(\rsv/\psv)$ that have the Harder--Narasimhan property.

	Finally we prove that $Z$ of the given form satisfies the local-finiteness property. If $Z\in V$ is a \textit{rational} point (i.e.\ $Z(\alpha)\in \Q+i\cdot \Q$ for all $\alpha\in \rootlattice$) then $Z\colon \sfK_0(\catDbC(\rsv))\to \C$ has discrete image, so the stability function $Z$ is locally finite by \cite[Lemma~4.4]{bridgelandStabilityConditionsK3}. Local finiteness of remaining points in $V$ follows from the density of rational points.
\end{proof}

\subsection{Braid group actions, geometrically}
\label{subsec:geometricbraidgroup}

The action $B_\sfex\circlearrowright \catDb(\ModPi)$ induces an action of $B_\sfex$ on $\catDb(\rsv)$ across the equivalence~\eqref{eq:tau}, we recall how this action arises naturally in the geometric context via spherical twists and tensor products with line bundles.

\begin{notation}
    When considering the action of $B$ on $\catDb(\rsv)$, we omit $\tau$ from the notation. Thus for example, $\R T_w$ denotes both the functor $\R T_w\colon \catDb(\ModPi)\to \catDb(\ModPi)$ defined in \S\ref{subsec:reflection-functors} and the induced functor $\tau^{-1}\circ \R T_w \circ \tau$ on $\catDb(\rsv)$.
\end{notation}

Recall that in any $\C$--linear $2$-Calabi--Yau triangulated category $\scrC$, an object $x$ induces two endofunctors $\bfT_x,\bfT'_x\colon\scrC\to\scrC$, called the \textit{twist} and \textit{cotwist} functors respectively, which act via
\begin{align}
	\begin{tikzcd}[ampersand replacement=\&]
		\bfT_x(y) \coloneqq \mathsf{Cone}\Big( \R\Hom(x, y)\otimes^\LL_\C x \ar{r}{\quad\ev\quad} \& y\Big)
	\end{tikzcd}
\end{align}
and
\begin{align}
	\begin{tikzcd}[ampersand replacement=\&]
		\bfT'_x(y) \coloneqq \mathsf{Cone}\Big(y  \ar{r}{\quad\ev\quad} \& \R\Hom(y,x)^\vee \otimes^\LL_\C x\Big)[-1]
	\end{tikzcd}\ .
\end{align}

\begin{definition}
	An object $x\in \scrC$ is said to be \textit{(2-)spherical} if $\R \Hom(x,x)\simeq\C\oplus\C[-2]$.
\end{definition}

\begin{remark}
	For a 2-spherical object, the functors $\bfT_x, \bfT'_x$ are equivalences and quasi-inverse to each other by \cite[Proposition~2.10]{ST01}.
\end{remark}

The category $\catDb(\rsv)$ is 2-Calabi--Yau and the simple perverse sheaves $\scrO_{C_i}(-1)[1]$ and $\scrO_C$ are spherical. The corresponding twist functors can in fact be identified with the reflection functors of \S\ref{subsec:reflection-functors}.

\begin{proposition}[{\cite[Proposition~\ref*{COHA-Yangian-prop:tilt_twist_thm}]{DPSSV-3}}]\label{prop:tilt_twist_thm}
    For each $i\in I$, the autoequivalence $\R T_i$ of $\catDb(\rsv)$ is isomorphic to the twist functor $\bfT_{\tau^{-1}(S_i)}$, while $\LL T_i$ is isomorphic to the cotwist functor $\bfT'_{\tau^{-1}(S_i)}$.
\end{proposition}
The action of the coweight lattice $\{1\}\ltimes \coweightlatticefin\subset B_\sfex$ is also described naturally once it is appropriately identified with the Picard group as we shall now describe.

\begin{lemma}
    For each $J\subset I_\sff$, the group of numerical $1$--cycles in $\psv$ relative to $\varpi_J$ is freely generated by the classes of (proper transforms of) $\{C_i\;\vert\; i\in J\}$. Consequently, the relative Picard group $\Pic(\psv/\ssv)$ is isomorphic to $\Z^J$.
\end{lemma}

In particular the Picard group $\Pic(\rsv)$ is isomorphic to $\Z^{I_\sff}$. An explicit basis dual to the curves $\{C_1,\ldots,C_e\}$ can be given by choosing divisors $D_1,\ldots,D_e\subset \rsv$ such that $\scrO_\rsv(D_i)=\det(\calP_i)$, where $\calP_i$ is the $i$th indecomposable summand of the tautological bundle on $\rsv$ as in \S\ref{subsec:kleinian-orbifolds}. Thus the intersection pairing is computed as $(\scrO_\rsv(D_i)\cdot \scrO_{C_i})\coloneqq(D_i\cdot C_j)=\delta_{ij}$. Viewing elements of $\Pic(\rsv)$ as additive functions on $\sfK_0(\catDbC(\rsv))\simeq \rootlattice$ in this fashion, we have the following observation.

\begin{proposition}\label{prop:pic-coweight}
	The identification $\sfK_0(\catDbC\rsv)\simeq \rootlattice$ induces a natural isomorphism
	\begin{align}\label{eq:Pic}
		\begin{tikzcd}[ampersand replacement=\&]
			\Pic(\rsv) \ar{r}{\sim} \& \coweightlatticefin
		\end{tikzcd}\ ,
	\end{align}
    where $\coweightlatticefin$ is identified with the hyperplane $\{\delta=0\}\subset \coweightlattice$ as in Formula~\eqref{eq:Lthetaconditions}. Under this isomorphism, the fundamental coweight $\Llambda_i$, for $i\in I_\sff$, corresponds to the line bundle $\scrO_\rsv(D_i)$, while the simple coroot $\Lalpha_i$ corresponds to $\scrO_\rsv(-C_i)$.
\end{proposition}

\begin{notation}
 	We denote by $\calL_{\Llambda}$ the line bundle associated to the coweight $\Llambda\in\coweightlatticefin$ under the isomorphism \eqref{eq:Pic}.
\end{notation}

The Picard group $\Pic(\rsv)$ acts on $\catDb(\rsv)$ via tensor products, and the action $\{1\}\ltimes \coweightlatticefin\circlearrowright\catDb(\rsv)$ induces this action across the identification of Proposition~\ref{prop:pic-coweight}.

\begin{proposition}[{\cite[Proposition~\ref*{COHA-Yangian-prop:braid}]{DPSSV-3}}]\label{prop:braid-tilting-intertwiner}
	For $\Llambda\in \coweightlatticefin$, The functor $\R L_\Llambda$ is naturally isomorphic to the equivalence $(\calL_\Llambda\otimes_\rsv -)$.
\end{proposition}

This completes the description of the $B_\sfex$ action on $\catDb(\rsv)$. Note that the spherical sheaves $\scrO_C$ and $\scrO_{C_i}(-1)$, as well as the generators $\scrO_\rsv(D_i)$ of $\Pic(\rsv)$ for $i=1, \ldots, e$, all admit an $A$-equivariant sheaf structure. Thus the above action $B_\sfex\circlearrowright\catDb(\rsv)$ commutes with that of $A$, in other words image of the induced map $B_\sfex\to \Aut(\catDb(\rsv))$ centralizes $A$.


\subsection{The stability manifold}\label{subsec:stability-manifold}

Referring the reader to \cite{B07} for definitions concerning stability conditions, we recall here the main result of \cite{Stab_Kleinian}, where a connected component $\Stab^\circ(\rsv)$ of the stability manifold of $\catDbC(\rsv)$ is computed (see also \cite{Thomas_Stab}).

Recall that a the points of the stability manifold $\Stab(\rsv)\coloneqq \Stab(\catDbC(\rsv))$ parametrise (locally finite) \textit{stability conditions} $\sigma=(Z,\calP)$ on $\catDbC(\rsv)$. Here the \textit{slicing} $\calP$ is a collection of full Abelian subcategories $\calP(t)\subset \catDbC(\rsv)$ indexed over $t\in \R$, and the \textit{central charge} $Z$ is a $\Z$-linear map $\sfK_0(\catDbC(\rsv))\to \C$ which can thus be identified with a point of $\coweightlattice\otimes \C$. The definition of the slicing ensures that for each unit interval $I\subset \R$, the full extension-closed subcategory $\calP(I)\coloneqq\left\langle \calP(t)\,\vert\, t\in I\right\rangle$ is the heart of a $t$-structure. In particular $\calP(0,1]$ is said to be the \textit{standard heart} of the stability condition $(Z,\calP)$, and for this standard heart $Z$ is a stability function satisfying the Harder--Narasimhan property.

Autoequivalences $\Phi\in\Aut(\catDbC(\rsv))$ naturally induce homeomorphisms $\Stab(\rsv)\to \Stab(\rsv)$ that we denote by the same symbol $\Phi$. Bridgeland's result, stated below, analyses such self-homeomorphisms of $\Stab(\rsv)$ by expressing it as a regular covering space of the open set
\begin{align}
	\frakh_\mathsf{reg}\coloneqq(\coweightlattice\otimes \C) \smallsetminus \bigcup_{\alpha\in \Delta}\{Z\,\vert\, Z(\alpha)=0\}\ .
\end{align}
Specifically we have the following.
\begin{theorem}[{\cite[Theorem~1.3]{Stab_Kleinian}}]\label{thm:Bridgelandcovering}
    Stability conditions with standard heart $\catPC(\rsv/\ssv)$ all lie in a connected component $\Stab^\circ(\rsv)\subset \Stab(\rsv)$. The forgetful map $\Stab(\rsv)\to \coweightlattice\otimes\C$ given by $(Z,P)\mapsto Z$ when restricted to this connected component yields a surjection
    \begin{align}
        \label{eqn:bridgelandmap}
        \Stab^\circ(\rsv)\longrightarrow \frakh_\mathsf{reg}/W
    \end{align}
    that is a regular covering map, with group of deck transformations given by the image of $B$ in $\Aut(\catDb(\rsv))$, and powers of the shift functor $[2]$.
\end{theorem}

Recall that by \cite[Proposition~5.3]{B07} stability conditions $(Z,\calP)$ on $\catDbC(\rsv)$ can be uniquely specified by giving the heart $\calP(0,1]$ and the stability function $Z$ on this heart satisfying the Harder--Narasimhan property.
\begin{lemma}
    For each $J\subset I_\sff$, if $Z:\sfK_0(\catDbC(\rsv))\to \C$ is a locally finite stability function on $\catPC(\rsv/\psv)$ with the Harder--Narasimhan property, then the corresponding stability condition $(Z,\calP)$ lies in Bridgeland's distinguished component $\Stab^\circ(\rsv)$.
\end{lemma}

\begin{proof}
    For $J=\emptyset$ the statement is tautological. For $J\subset I_\sff$ non-empty, Theorem~\ref{thm:perverse-stabfuncs} shows that the forgetful map $\Stab(\rsv)\to \coweightlattice\otimes \C$ maps the subset $\{(Z,\calP)\,\vert\, \calP(0,1]=\catPC(\rsv/\psv)\}\subset \Stab(\rsv)$ homeomorphically onto $V_J\coloneqq-\calD_J+i\calC_J^0$. Now $V_J$ is connected and contains the subset $-\calC^++i\calC_J^0$, so any point $(Z,\calP)$ with $\calP(0,1]=\catPC(\rsv/\psv)$ is connected in $\Stab(\rsv)$ to a point $(Z', \calP')$ with the same standard heart, but $Z' =-\Ltheta + i\Lomega$ for $\Lomega$ generic in $\calC_J^0$, and $\Ltheta$ generic in $\calC^+$.
    
    A straightforward check shows that each of the sheaves $\scrO_{C_i}(-1)$ for $i\in I_\sff$ and $\scrO_{C}$ are all $(Z',\calP')$-stable, and the simple objects of $\catPC(\rsv/\ssv)$ thus all lie in $\calP'(\half, 3/2]$. It
    follows that $(Z',\calP')$ can be rotated (via the action $\C\circlearrowright \Stab(\rsv)$, seen $\C$ as a subgroup of the universal cover of the group $\GL^+(2,\R)$ and the action of the latter is defined in \cite[Lemma~8.2]{B07}) to a point of $\Stab^\circ (\rsv)$. The result follows.
\end{proof}

\begin{remark}
    The manifold $\Stab(\rsv)$ is known to be connected when $\ssv$ is of type $\sfA_n$, as proved in \cite{ishiiStabilityConditionsAnSingularities}.
\end{remark}

Consider one of the above stability conditions $(Z,\calP)\in \Stab^\circ(\rsv)$ with $\calP(0,1]=\catPC(\rsv/\psv)$ ($J\neq \emptyset$) and $Z=-\Ltheta + i\Lomega$. The semistable objects of phase $1$ can be computed as
\begin{align}
    \calP(1)&=\{x\in \catPC(\rsv/\psv)\,\vert\, \Lomega[x]=0\} \\
            &=\left\langle \catP_{C_{J_1}}(\rsv/\ssv) \cup \ldots\cup \catP_{C_{J_n}}(\rsv/\ssv)\cup \{\scrO_p\,\vert\, p\in C\} \right\rangle\ ,
\end{align}
where $J_1,\ldots,J_n$ are connected components of $I_\sff\smallsetminus J$ and $\langle\cdot\rangle$ denotes extension-closure. In particular the torsion pair $(T_J,F_J)\coloneq(\calP(1), \calP(0,1))$ on $\catPC(\rsv/\psv)$, and hence the tilted heart $\calP[0,1)$, is independent of the stability condition chosen. This is the \textit{`reversed' (semi-)geometric heart} considered in \cite[\S5.5]{Shimpi_Torsion_pairs}.

\begin{definition}
    For $J\subset I_\sff$ as above, write $\barcatPC(\rsv/\psv)$ for the heart obtained by tilting $\catPC(\rsv/\psv)$ in the torsion class $T_J$, i.e.\
	\begin{align}
		\barcatPC(\rsv/\psv) \coloneqq\left\langle T_J[-1]\cup \{x\in\catPC(\rsv/\psv)\,\vert\, \Hom(t,x)=0\text{ for all }x\in T_J\} \right\rangle\ .
	\end{align}
    If $J=I_\sff$ we also write $\barcatCoh(\rsv)$ for the category $\barcatPC(\rsv/\psv)$.
\end{definition}

\begin{theorem}\label{thm:allstabs}
    Given any stability condition $(Z,\calP)\in \Stab^\circ(\rsv)$ and any unit interval $\calI\subset \R$ such that $\calP(\calI)$ is Abelian, there is a $J\subset I_\sff$ and an element $b\in B$ such that the heart $b\cdot \calP(\calI)$ is, up to shift, one of $\catPC(\rsv/\psv)$ or $\barcatPC(\rsv/\psv)$.
\end{theorem}

\begin{proof}
    Noting that the actions of $B$ and $\C$ on the stability manifold commute, we may rotate the stability condition by $e^{i\pi\cdot \text{min}(\calI)}$ and assume $\calI\subseteq [0,1]$. Write $Z=-\Ltheta+i\Lomega$ for real coweights $\Ltheta,\Lomega\in\coweightlattice\otimes R$.

    If $\Lomega(\delta)\neq 0$, then up replacing $(Z,\calP)$ by its shift we may assume $\Lomega(\delta)>0$, and so up to the action of some element of $B$, the coweight $\Lomega$ lies in the Weyl chamber $\calC^+$. Let $J\subset I$ denote the indices $i$ where $\Lomega(\alpha_i)=0$, and note that $\Lomega$ is invariant under the action of the parabolic subgroup $W_J=\langle  s_i\,\vert\, i\in J \rangle$. Furthermore, $\Lomega(\delta)\neq 0$ implies that $J$ is a proper subset of $I_\sff$, so that the parabolic subgroup above is of finite type. Now $Z\in\frakh_\mathsf{reg}$, so we necessarily have $\Ltheta(\alpha_i)\neq 0$ for all $i\in J$. It follows that the action of some element $w\in W_J$ (and hence some element of $B$) moves $\Ltheta$ into the chamber $\bigcap_J\{\alpha_i>0\}$ while leaving $\Lomega$ invariant in $\calC^+$.

    On the other hand, $Z$ is clearly a stability function on $\catPC(\rsv/\ssv)$, i.e.\ there is a stability condition $(Z,\calP')$ with $\calP'(0,1]=\catPC(\rsv/\ssv)$ that lies in the same fibre as $(Z,\calP)$ under Bridgeland's regular covering map \eqref{eqn:bridgelandmap}. Since the deck group of a regular cover acts transitively on fibres, we see that $\calP=\calP'$ up to even shifts and the action of $B$. If $\calI=(0,1]$ then we are done, if $\calI=[0,1)$ then we could instead move $\Ltheta$ into the chamber $\bigcap_J\{\alpha_i<0\}$ and continue as before to again conclude $\calP(\calI)=\catPC(\rsv/\ssv)$.
    
    \medskip

    If $\Lomega(\delta)=0$, we again manipulate $Z$ via shifts and the $B$-action till it becomes a stability function on $\catPC(\rsv/\psv)$ for some $J$. Since $Z\in \frakh_\mathsf{reg}$, we necessarily have $\Ltheta(\delta)\neq 0$ and hence, up to shift, $\Ltheta(\delta)>0$. Then as before, we can begin by moving $\Lomega$ into the chamber $\calC^0$, and let $J$ be such that $\calC^0_J$ is the smallest face of $\calC^0$ containing $\Lomega$. Thus $J=\{i\in I_\sff\,\vert\,\Lomega(\alpha_i)=0\}$. The $\Ltheta$ coordinate is then taken care of by the stabiliser $W(J)$ of $\calC_J^0$, which by Lemma~\ref{lem:fundamentaldomain-WJ} can move $\Ltheta$ into $\calD_J$ whilst leaving $\Lomega$ invariant. Thus up to shifts and the action of $B$, we see that $Z$ lies in the region $V_J=-\calD_J+i\calC_J^0$, and is hence a stability function on $\catPC(\rsv/\psv)$. The result follows, with $\calP(\calI)$ lying in the $B$-orbit of $\catPC(\rsv/\psv)$ or $\barcatPC(\rsv/\psv)$ depending on whether $\calI=(0,1]$ or $\calI=[0,1)$ holds.
\end{proof}

\section{Numerics of $t$-structure variation}\label{sec:numerical-tilts}

Theorem~\ref{thm:allstabs} shows that all hearts in $\catDbC(\rsv)$, that admit stability conditions in $\Stab^\circ(\rsv)$, arise from the Kleinian orbifolds $\calX_J$ with $J\subseteq I_\sff$. The \textit{cohomological Hall algebra} of $\catCoh(\calX_\emptyset)\simeq \nilpPi$ has been the object of extensive study since its introduction in \cite{SV_Cherednik} (see, for instance, \cite{SV_Yangians=COHA} and the references therein), and our strategy to compute the cohomological Hall algebras associated to the remaining hearts is to approximate them by the $\Aut(\catDb(\rsv))$-orbit of $\nilpPi$. More specifically we will consider the action of $\Pic(\rsv)$, which can be identified with $\coweightlatticefin$ as in Proposition~\ref{prop:pic-coweight}.

\begin{definition}
    Given $J\subseteq I_\sff$, we say a line bundle $\calL\in \Pic(\rsv)$ is \textit{$\varpi_J$-ample} if it is the pullback (under $\pi_J$) of some $\varpi_J$-ample line bundle on $\psv$, i.e.\ if $(\calL\cdot C_i)\geq 0$ for all $i\in I_\sff$ with equality holding if and only if $i\in \Jcomp$.
\end{definition}

Fix a $\varpi_J$-ample line bundle $\calL$ with corresponding coweight $\Llambda$, which we use to define two stability functions -- first, as in Definition~\ref{def:slopefunction}, the $\Llambda$-slope function
$\mu_\Llambda\colon \nilpPi\smallsetminus \{0\} \to \R$. Second, a Bridgeland stability condition $(Z_\Llambda,\calP_\Llambda)$ determined by its standard heart and charge
\begin{align}
	\calP_\Llambda(0,1]=\catPC(\rsv/\ssv)\quad\text{and}\quad Z_\Llambda = \Llambda + i\cdot w_0\Lrho \; \in\; \calC_J^0+i\cdot \calC^+\ ,
\end{align}
where $\Lrho\coloneqq\sum_{i\in I} \Lomega_i$ is the sum of fundamental coweights.

This section proves the following key result, generalising \cite[Theorem~\ref*{COHA-Yangian-thm:tau-slicing}--(\ref*{COHA-Yangian-thm:tau-slicing-2})]{DPSSV-3}.
\begin{theorem}\label{thm:slicing}
    For $J\subseteq I_\sff$ and $\varpi_J$-ample line bundle $\calL\in \Pic(\rsv)$ corresponding to the coweight $\Llambda\in \coweightlatticefin$, the
    stability condition $(Z_\Llambda,\calP_\Llambda)$ defined above satisfies
	\begin{align}	
   		\calP_\Llambda(-1/2, 1/2] = \catPC(\rsv/\psv)\quad\text{and}\quad\calP_\Llambda[-1/2,1/2)=\barcatPC(\rsv/\psv)\ .
	\end{align}
    Further, consider the bi-infinite sequence of real numbers 
    \begin{align}
    	t_n\coloneqq \frac{1}{\pi}\arctan(nh)\ ,
    \end{align}
	where $h$ is the Coxeter number for the Dynkin diagram $\Delta$, which converges to $t_{\pm\infty}\coloneqq\pm 1/2$ in the limit $n\to \pm \infty$. For any $n\in \Z$ the $A$-equivariant functor
    \begin{align}\label{eq:tk}
    	\tau^{-1}\circ (\R L_\Llambda)^{-n}\colon \catDb(\nilpPi)\longrightarrow \catDbC(\rsv)
    \end{align}
    restricts to the equivalences
	\begin{align}
		\begin{tikzcd}[ampersand replacement=\&, row sep=tiny]
            \nilp^{\leq 0}\,\Pi \ar{r}{\sim}\& \calP_\Llambda(-t_n,\phantom{-}\half] \\
			\nilp^{> 0}\,\Pi \ar{r}{\sim}\&\calP_\Llambda(\half,1-t_n] \\
            \nilp^{\phantom{<0}}\,\Pi\ar{r}{\sim}\&\calP_\Llambda(-t_n,1-t_n]
		\end{tikzcd}
	\end{align}
    where the Harder--Narasimhan strata $\nilp^{\geq 0}\Pi, \nilp^{<0}\Pi$ are with respect to the slope function $\mu_{\Llambda}$ as above.
\end{theorem}
The proof uses and builds upon an interface between
Broomhead--Pauksztello--Ploog--Woolf's \textit{Heart fans}~\cite{HeartFan} and
Bridgeland stability conditions.

\subsection{Heart fans and stability arcs}\label{subsec:stabilityarc}

Hearts of bounded $t$-structures on a triangulated category $\scrC$ form a partially ordered set, where the partial order is determined by containment of \textit{coaisles}: given two hearts $H, K$ of bounded $t$-structures on $\scrC$, we have
\begin{align}
	H\leq K \quad \iff \quad H[\leq 0] \subseteq K[\leq 0]\ .
\end{align}
Elements of the interval $[H[-1],H]$ are called \textit{hearts of intermediate $t$-structures} (or simply \textit{intermediate hearts}) with respect to $H$. 
Fix a bounded heart $H\subset \scrC$, and suppose $H$ is Artinian and Noetherian with finitely many simple objects. These conditions ensure $\sfK_0(H)\simeq \sfK_0(\scrC)$ is a free Abelian group of finite rank, so that $\Theta\coloneqq\Hom(\sfK_0(\scrC),\R)$ is a finite dimensional $\R$-vector space. The central construction of \cite{HeartFan}, the \textit{heart fan} of $H$, is a convex-geometric ensemble in $\Theta$ that facilitates the analysis of the interval $[H[-1],H]$.

\begin{definition}
    Given the heart of an intermediate $t$-structure $K\in [H[-1],H]$, the \textit{heart cone} of $K$ is the closed convex cone $\calC(K)\subset \Theta$ defined as
	\begin{align}
		\calC(K) \coloneqq \left\{\theta\in \Theta\;\middle\vert\; \theta(k)\geq 0\quad\text{for all }k\in K\right\}\ .
	\end{align}
    We say $C\subseteq\Theta$ is an \textit{intermediate heart cone} if $C=\calC(K)$ for some $K\in [H[-1],H]$.
\end{definition}

\begin{theorem}[{\cite[Theorem~A, Corollary~3.3]{HeartFan}}]
    For $H$ as above, the set
	\begin{align}
		\mathsf{HFan}(H)\coloneqq \smashoperator[r]{\bigcup_{K\in [H[-1],H]}}\; \mathsf{faces}(\calC(K))
	\end{align}
    of all faces of intermediate heart cones is a complete simplicial fan in $\Theta$, called the \textit{heart fan of $H$}.

    Moreover, for any $\theta\in \Theta$, the subset $\{K\,\vert\, \theta\in \calC(K)\}\subseteq  [H[-1],H]$ of intermediate hearts whose heart cone contains $\theta$ is an \textit{interval} of the form $\{K\,\vert\, H_\theta\leq K\leq H^\theta\}$, with maximum and minimum elements given by
    \begin{align}
        H_\theta &=
        \left\langle
                \left\{ x\in H
                    \;\middle\vert\; \begin{array}{l}\theta(s)>0 \text{ for all}\\
                \text{sub-objects }s\hookrightarrow x\end{array}
                \right\} \;\cup\;
                \left\{ x[-1]\in H[-1]
                    \;\middle\vert\; \begin{array}{l}\theta(f)\leq 0 \text{ for all}\\
                \text{factors }x\twoheadrightarrow f\end{array}
                \right\}
        \right\rangle\ ,\\[6pt]
        H^\theta &=
        \left\langle
                \left\{ x\in H
                    \;\middle\vert\; \begin{array}{l}\theta(s)\geq 0 \text{ for all}\\
                \text{sub-objects }s\hookrightarrow x\end{array}
                \right\} \;\cup\;
                \left\{ x[-1]\in H[-1]
                    \;\middle\vert\; \begin{array}{l}\theta(f) < 0 \text{ for all}\\
                \text{factors }x\twoheadrightarrow f\end{array}
                \right\}
        \right\rangle\ .
    \end{align}
    These hearts $H_\theta,H^\theta$ remain constant as $\theta$ varies in the relative interior of an intermediate heart cone.
\end{theorem}

We use the heart fan to study families of Bridgeland stability functions on $H$. Consider a pair of non-zero vectors $\theta_0,\theta_{\half}\in \Theta$ such that $\theta_0$ lies in the (relative) interior of $\sfC(H)$. The conditions guarantee that
\begin{align}
	Z\coloneqq \theta_{\half} + i\cdot \theta_0 \, \in \Theta\otimes\C
\end{align}
is a stability function on $H$, and hence by \cite[Proposition~5.3]{B07} we have a Bridgeland stability condition $(Z,\calP)$ on $\scrC$ with $\calP(0,1]=H$.

Following \cite[\S6.2]{HeartFan}, the slicing $\calP$ can be explicitly described by extending the pair of vectors $\theta_0,\theta_{\half}$ to an arc $[0,1]\to \Theta$ given by
\begin{align}\label{eq:thetatarc}
    \theta_t \coloneqq \cos(\pi t)\,\theta_0 + \sin(\pi t)\,\theta_{\half}\ .
\end{align}

\begin{lemma}
    For $(Z,\calP)$ as above and $t\in (0,1]$, we have
    \begin{align}
        \calP(1-t)
        &= H_{\theta_t}[-1]\cap H^{\theta_t}\\
        &= \left\{
            h\in H \;\middle\vert\;
            \begin{array}{l}
                \theta_t(x)=0\ ,\text{ and }\theta_t(s)\geq 0\\[2pt]
                \text{for all sub-objects }s\hookrightarrow x
            \end{array}
        \right\}\ .
    \end{align}
\end{lemma}

\begin{proof}
	Begin by observing that $\theta_t\colon \sfK_0(\scrC)\to \R$ is the functional $w_t\circ Z$, where $w_t\colon \C\to \R$ is the unique linear map satisfying
	\begin{align}
		w_t(e^{i\pi(1-t)})=0\quad\text{and} \quad w_t(ie^{i\pi(1-t)})=-1\ .
	\end{align}
	Now given $x\in H $, since $Z$ is a stability function on $H$ we see that $Z(x)$ has phase in the range $(0,1]$ and hence
	\begin{align}
		\arg(Z(x)) &\leq 1-t \quad \text{ if and only if }
		w_t\left(Z(x)\right)\geq 0\ , \text{ i.e. } \theta_t(x)\geq 0\ ,
		\\
		\arg(Z(x)) &\geq 1-t \quad \text{ if and only if }
		w_t\left(Z(x)\right)\leq 0\ , \text{ i.e. } \theta_t(x)\leq 0\ .
	\end{align}
	The result follows, since $x$ lies in $\calP(1-t)$ if and only if $\arg(Z(x))=1-t$ and $\arg(Z(s))\leq 1-t$ for all sub-objects $s\hookrightarrow x$.
\end{proof}

\begin{lemma}\label{lem:bridgeland-to-king}
	For $(Z,\calP)$ as above and $t\in (0,1]$, the heart $K=\calP(-t,1-t]$ satisfies $\theta_t\in \calC(K)$. In fact, we have $K=H^{\theta_t}$, i.e.\ $K$ is the maximal among intermediate hearts whose heart cone contains $\theta_t$. Likewise, the heart $K=\calP[-t,1-t)$ is equal to $H_{\theta_t}$, the minimal heart whose heart cone contains $\theta_t$.
\end{lemma}

\begin{proof}
	Note we can write $K$ as the extension-closure of $\calP(0,1-t]$ and $\calP(-t,0]=\calP(1-t,1][-1]$. Evidently, we have $\theta_t(x)\geq 0$ whenever $x\in \calP(0,1-t]$ and $\theta_t(x)< 0$ whenever $x\in \calP(1-t,1]$. It follows that $\theta_t$ lies in $\calC(K)$.

	Since $K$ also contains the subcategory $\calP(1-t)=H_{\theta_t}[-1]\cap H^{\theta_t}$, we must also have $K=\sfH^{\theta_t}$ by \cite[Corollary~2.14]{Shimpi_Torsion_pairs}. The statement for $H_{\theta_t}$ is analogous.
\end{proof}

\subsection{The heart fan of $\nilpPi$}

In the category $\catDbC(\rsv)$ we fix the standard heart $\sfH\coloneqq \catPC(\rsv/\ssv)$ that is Artinian and Noetherian, and recall the construction of its heart fan following \cite{Shimpi_Torsion_pairs}. By Theorem~\ref{thm:VdBmainthm} the heart $\catP(\rsv/\ssv)$ is a positive tilt of $\catCoh(\rsv)$, and hence the heart $\catCohC(\rsv)\subset \catDbC(\rsv)$ is intermediate with respect to $\sfH$. Likewise, the torsion pairs introduced in Formula~\eqref{eq:pervtorsion} satisfy $\calF_{J}\subseteq \calF_{\Jcomp}$ whenever $\Jcomp\subseteq J\subseteq I_\sff$, consequently we have the inequalities
\begin{align}
    \sfH[-1] \,\leq\, \barcatPC(\rsv/X_{\Jcomp}) \,\leq\, \barcatPC(\rsv/\psv) \,\leq\, \catPC(\rsv/\psv) \,\leq\, \catPC(\rsv/\rsv_{\Jcomp}) \,\leq\, \sfH\ .
\end{align}
In particular, both categories $\catPC(\rsv/\psv)$ and $\barcatPC(\rsv/\psv)$ are intermediate with respect to $\sfH$.

More intermediate hearts can be enumerated by examining the action of $\Aut(\catDbC(\rsv))$ on the category, in particular we look at the subgroups $B,\coweightlatticefin\subset B_\sfex$ which together generate $B_\sfex$ and give rise to functors of algebraic and geometric significance respectively.
\begin{theorem}[{\cite[Theorem 6.5]{IR08}}]
    Given an element $b\in B$  acting via the functor $\beta\in \Aut(\catDb(\rsv))$, we have the following.
    \begin{enumerate}\itemsep0.2cm
        \item The heart $\beta(\sfH)$ is intermediate with respect to $\sfH$ if and only if there is an element $w\in W$ with $b=T_w$, equivalently $\beta=\R T_w$.
        \item The heart $\beta(\sfH)[-1]$ is intermediate with respect to $\sfH$ if and only if there is an element $w\in W$ with $b=(T_w)^{-1}$, equivalently $\beta=\LL T_w $.
    \end{enumerate}
    The above hearts respect the Bruhat order on $W$, i.e.\ if $w=uv$ with $\ell(w)=\ell(u)+\ell(v)$ then
	\begin{align}
	   \R T_u (\sfH) \,\geq\, \R T_w (\sfH) \quad \text{and} \quad \LL  T_w (\sfH)[-1] \,\geq\, \LL  T_u (\sfH)[-1]\ .
\end{align}
\end{theorem}

\begin{theorem}[{\cite[Theorem~6.3--1]{Shimpi_Torsion_pairs}}]\label{thm:linebundle-tstructure}
    Given an element $\Llambda\in \coweightlatticefin$ acting via the line bundle twist $(\calL_{\Llambda}\otimes -)$, the following statements are equivalent.
    \begin{enumerate}\itemsep0.2cm
        \item The heart $\calL_{\Llambda}^\vee\otimes \sfH$ is intermediate with respect to $\sfH$.
        \item The heart $\calL_{\Llambda}\otimes \sfH[-1]$ is intermediate with respect to $H$.
        \item The bundle $\calL_{\Llambda}$ is nef, equivalently $\Llambda$ is dominant.
    \end{enumerate}
\end{theorem}

We recall the calculation of $\mathsf{HFan}(\sfH)$, which under the identification $\Hom(\sfK_0(\catDbC\rsv),\R)\simeq \coweightlattice\otimes \R$ coincides with the fan underlying the Tits cone decomposition \eqref{eq:titsdecomp}.
\begin{theorem}[{\cite[Theorem~C]{Shimpi_Torsion_pairs}}]\label{thm:heartfan-of-nilpPi}
    Given a closed convex cone $\calC\subseteq \coweightlattice\otimes\R$ and a vector $\theta$ in its relative interior, $\calC$ is an intermediate heart cone for $\sfH\coloneqq\catPC(\rsv/\ssv)$ if and only if one of the following holds.
    \begin{enumerate}\itemsep0.2cm
        \item The cone $\calC$ is of the form $w\calC^+$ for some $w\in W$. In this case, $\sfH^\theta=\R T_w (\sfH)$ is the unique intermediate heart with heart cone $\calC$.
        \item The cone $\calC$ is of the form $w^{-1}\calC^-$ for some $w\in W$. In this case, $\sfH^\theta=\LL T_w (\sfH)[-1]$ is the unique intermediate heart with heart cone $\calC$.
        \item The cone $\calC$ is a non-zero face of $w\calC^0$ for some $w\in W_\sff$. Such cones are of the form $w\calC^0_J$ for $J\subseteq I_\sff$ non-empty. Choosing $(w,J)$ such that $\calC=w\calC^0_J$ and $w$ has minimal length among all such representatives, the maximal and minimal hearts with heart cone $\calC$ are given as
		\begin{align}
	        \sfH^\theta=\R T_w(\catPC(\rsv/\psv))\quad\text{and}\quad \sfH_\theta=\R T_w(\barcatPC(\rsv/\psv))
		\end{align}
        where $\theta\in w\calC^0_J$ is a generic vector.
    \end{enumerate}
\end{theorem}

\begin{remark}
	If $J\subseteq I_\sff$ is non-empty, then for any $\varpi_J$-ample line bundle $\calL$ with corresponding coweight $\Llambda$ (seen as a vector in $\coweightlattice\otimes \R$ under the identifications $\Pic(\rsv)\simeq \coweightlatticefin\hookrightarrow \coweightlattice$) we have
    \begin{align}
    \sfH^{\Llambda} = \catPC(\rsv/\psv)\quad\text{and}\quad \sfH_{\Llambda} = \barcatPC(\rsv/\psv) \ .
	\end{align}
    In particular if $\calL$ is $\pi$--ample, then $\sfH^{\Llambda}=\catCohC(\rsv)$.
\end{remark}

\subsection{Harder--Narasimhan strata revisited}\label{subsec:HN-revisited}

We now prove Theorem~\ref{thm:slicing} by considering appropriate stability arcs in the heart fan of $\sfH$. Accordingly, fix a non-empty subset $J\subseteq I_\sff$ and a coweight $\Llambda\in \coweightlatticefin\subset \coweightlattice$ such that the associated line bundle $\calL_{\Llambda}$ is $\varpi_J$-ample. The stability condition $(Z_\Llambda, P_\Llambda)$ is (by definition) constructed from the stability function $Z_\Llambda=\theta_{\half}+i\cdot\theta_0$ on $\sfH$, where
\begin{align}
	\theta_0\coloneqq w_0(\Lrho) \quad\text{and}\quad \theta_{\half} \coloneqq \Llambda \quad \in \coweightlattice\otimes\R\ .
\end{align}

\begin{lemma}\label{lem:bridgeland-to-slope}
    For $t\in (0,1]$, the objects of $\calP_\Llambda(1-t)$ are precisely of the form $\tau^{-1}(M)$ where $M\in \nilpPi$ is a $\mu_\Llambda$--semistable module with slope $\cot(\pi t)$. In particular, the $A$-equivariant functor
    $\tau$ restricts to equivalences
	\begin{align}
		\begin{tikzcd}[ampersand replacement=\&]
				\nilp^{\leq 0}\,\Pi\ar{r}{\sim}\& \calP_\Llambda (0,\half]
		\end{tikzcd}
	\quad\text{and}\quad
	\begin{tikzcd}[ampersand replacement=\&]
		\nilp^{> 0}\,\Pi\ar{r}{\sim}\& \calP_\Llambda (\half,1]
	\end{tikzcd}\ .
	\end{align}
\end{lemma}

\begin{proof}
	Let $M$ be a nilpotent $\Pi$--module, and observe that $(w_0\Lrho, [M])=\dim(M)$. Since the coweight $\Llambda$ lies in the hyperplane $\coweightlatticefin\otimes \R$, we have $w_0(\Llambda)=-\Llambda$. Thus the slope of $M$ may be computed as
	\begin{align}
        \mu_\Llambda(M)=-\,\frac{(\theta_{\half}, [M])}{(\theta_{0\phantom{/2}}, [M])}\ .
	\end{align}
	Defining $\theta_t$ as in Formula~\eqref{eq:thetatarc}, we see that if $t\in (0,1]$ is such that $\mu_\Llambda(M)=\cot(\pi t)$, then $(\theta_t, [M])=0$. Further, for any submodule $N\subset M$ we have $\mu_\Llambda(N)\leq \mu_\Llambda(M)$ if and only if $(\theta_t,[N])\geq 0$, i.e.\ $M$ is $\mu_\Llambda$--semistable if and only if $\tau^{-1}(M)$ lies in $\sfH_{\theta_t}[-1]\cap \sfH^{\theta_t}=\calP_\Llambda(1-t)$. The result follows.
\end{proof}

We can then prove the key result of this section.
\begin{proof}[Proof of Theorem~\ref{thm:slicing}]
    Defining the arc $(\theta_t\,\vert\, 0<t\leq 1)$ as in Formula~\eqref{eq:thetatarc}, it is clear from Lemma~\ref{lem:bridgeland-to-king} and Theorem~\ref{thm:heartfan-of-nilpPi} that $\calP_\Llambda(-1/2,1/2]=\catPC(\rsv/\psv)$ and $\calP_\Llambda[-1/2,1/2)=\barcatPC(\rsv/\psv)$.

    For brevity write $\vartheta_n\coloneqq\theta_{t_n}$ when $n\geq 0$, and $\vartheta_n\coloneqq \theta_{1+t_n}$ when $n<0$.

	Note that $(\R L_\Llambda)^{-n}$ acts on $\sfK_0(\nilpPi)\simeq \rootlattice$ via $\ell_\Llambda^{-n}\in W$ (see Formula~\eqref{eq:ell_Llambda}), and thus on the coweight lattice via the inverse--transpose
    \begin{align}
		\ell_\Llambda^{-n}(\theta)= \theta + (\theta,\delta)\cdot n\Llambda\ .
	\end{align}
	In particular we see that if $n\geq 0$ then $\vartheta_n$ lies in the ray spanned by $\ell_\Llambda^{-n}(\vartheta_0)$, and hence in the interior of the intermediate heart cone $\ell_\Llambda^{-n}(\calC^+)$. On the other hand if $n<0$, then $\vartheta_n$ lies in the ray spanned by $\ell_{\Llambda}^{-n}(-\vartheta_0)$, i.e.\ in the intermediate heart cone $\ell_\Llambda^{-n}(\calC^-)$.

    It then follows from Theorem~\ref{thm:heartfan-of-nilpPi} and Lemma~\ref{lem:bridgeland-to-king} that
	\begin{align}
        \calP_\Llambda(-t_n,1-t_n]
        &= (\calL_\Llambda)^{-n}\otimes \sfH & \text{ if }n\geq 0\ , \\
        \calP_\Llambda(-1-t_n,-t_n] &=
        (\calL_\Llambda)^{-n}\otimes \sfH[-1] & \text{ if }n<0\ .
	\end{align}
	In other words, $\tau^{-1}\circ(\R L_\Llambda)^{-n}$ is an equivalence between $\nilpPi$ and $\calP_\Llambda(-t_n,1-t_n]$ for all $n\in \Z$, as required.

	To show that this restricts to the given equivalence on $\nilp^{>0}\,\Pi$, we note that $\sfK\coloneqq\catPC(\rsv/\psv)$ is invariant under the action of the $\varpi_J$-ample bundle $\calL_\Llambda$ and hence
	\begin{align}
		\calP_\Llambda (\half,1-t_n] =\underbrace{\calP_\Llambda
        (-\half,\half]\;[1]}_{\sfK[1]=(\calL_\Llambda)^{-n}\otimes \sfK[1]}\;\cap\;
		\underbrace{\calP_\Llambda
        (-t_n,1-t_n]}_{(\calL_\Llambda)^{-n}\otimes\sfH}	=
        (\calL_\Llambda)^{-n}\otimes(\underbrace{\calP_\Llambda(\half,1]}_{\sfK[1]\cap \sfH})\ .
	\end{align}
	Thus Lemma~\ref{lem:bridgeland-to-slope} yields $\calP_\Llambda(\half,1-t_n]=\tau^{-1}\circ (\R L_\Llambda)^{-n}(\nilp^{>0}\,\Pi)$ as required. The corresponding statement for $\nilp^{\leq 0}\,\Pi$ is proved similarly.
\end{proof}

\begin{remark}
    A consequence of Theorem~\ref{thm:slicing} is that the coaisle corresponding to the heart $\catPC(\rsv/\psv)$ can be obtained as the intersection
	\begin{align}
		\calP_\Llambda(-\infty, 1/2]=\bigcap_{n\geq 0}\calP_\Llambda(-\infty, 1-t_n]
	\end{align}
    of coailses corresponding to the hearts $(\calL_\Llambda)^{-n}\otimes \sfH$. In other words in the poset of all $t$-structures on $\catDbC(\rsv)$, the heart $\catPC(\rsv/\psv)$ is the \textit{infimum} of the decreasing sequence
	\begin{align}
		\sfH > (\calL_\Llambda)^{-1}\otimes \sfH>(\calL_\Llambda)^{-2}\otimes \sfH> \ldots
	\end{align}
    Likewise, one can show that $\barcatPC(\rsv/\psv)$ is the \textit{supremum} of the increasing sequence
	\begin{align}
		\sfH[-1]< (\calL_\Llambda)\otimes \sfH[-1] < (\calL_\Llambda)^2\otimes \sfH[-1] < \ldots
	\end{align}
    This gives an alternate proof of \cite[Theorem~6.3--(2)]{Shimpi_Torsion_pairs}.
\end{remark}

\section{Limiting COHAs}\label{sec:limitingCOHAs}

In this section, we introduce a cohomological Hall algebra structure on the $A$-equivariant Borel--Moore homology of the moduli stack of objects belonging to $\sfH^{\Llambda}$. This construction is performed by using the framework of \textit{limiting COHAs} introduced in \cite[Part~\ref*{COHA-Yangian-part:COHA-stability-condition}]{DPSSV-3}.

\medskip

We fix a subset $J\subseteq I_\sff$ and a coweight $\Llambda\in \coweightlatticefin$ such that the associated line bundle $\calL_{\Llambda}$ is $\varpi_J$-ample on $\rsv$. Recall that we have identified $\Z I=\sfK_0(\nilpPi)$ and $\rootlattice$ via \eqref{eq:simpleclasses}.

\subsection{Moduli stack of nilpotent representations}

\subsubsection{Preliminaries}

We denote by $\stackRep(\Pi_\qv)$ the classical moduli stack parametrising finite-dimensional representations of $\Pi_\qv$. It splits as a disjoint union
\begin{align}
	\stackRep(\Pi_\qv)=\bigsqcup_\bfd\stackRep_\bfd(\Pi_\qv)
\end{align}
into closed and open connected components, according to the dimension vector $\bfd \in \N I$. Each $\stackRep_\bfd(\Pi_\qv)$ is a finite type classical geometric stack. We let $\Lambda_\qv$ stand for the closed substack of $\stackRep(\Pi_\qv)$ parametrising nilpotent representations. Thus,
\begin{align}
	\Lambda_\qv=\bigsqcup_{\bfd \in \N I} \Lambda_{\bfd}\ .
\end{align}
This is a finite type classical geometric stack. Furthermore,  is pure of dimension
\begin{align}
	\dim \Lambda_{\bfd}=-\langle \bfd, \bfd \rangle\ .
\end{align}

There exists a derived enhancement $\dstackRep_\bfd(\Pi_\qv)$ of $\stackRep_\bfd(\Pi_\qv)$ (cf.\ \cite[\S2.1.4]{VV_KHA} or \cite[\S~I.2]{DPS_Torsion-pairs}). Since the category of representations of $\Pi_\qv$ is 2-Calabi-Yau, $\dstackRep(\Pi_\qv)$ is a derived lci geometric derived stack of finite type over $\C$.
\begin{definition}
	The \textit{derived} moduli stack $\dLambda_\qv$ of nilpotent finite-dimensional representations of $\Pi_\qv$ is the formal completion of $\Lambda_\qv$ inside $\dstackRep(\Pi_\qv)$.
\end{definition}

Let $\Ltheta \in \coweightlattice\otimes \R$ be a stability condition. The subfunctor of $\Ltheta$-semistable nilpotent $\Pi_\qv$-representations of dimension $\bfd\in \rootlattice$ forms an open substack $\Lambda_\bfd^{\Ltheta\textrm{-}\mathsf{ss}}$ of $\Lambda_\bfd$. The latter admits a canonical enhancement\footnote{The construction of a canonical derived enhancement of an open embedding of a geometric classical stack into a geometric derived stack follows from \cite[Proposition~2.1]{STV}.}, so there is also a derived open substack  $\dLambda_\bfd^{\Ltheta\textrm{-}\mathsf{ss}}$ of $\dLambda_\bfd$. 
\subsubsection{Moduli stacks of Harder-Narasimhan strata}

Following Formula~\eqref{eq:tk}, for $k \in \Z$ we set $\nu_k\coloneqq t_{2k}$. Let $\calP_{\Llambda}$ be the slicing associated to $\Llambda$, which is introduced in \S\ref{subsec:HN-revisited}. Let us denote by $\tau_k$ and $\tau_{\sfrac{1}{2}}$ the $t$-structure on $\catDb(\nilpPi)$ having heart $\calP_{\Llambda}(\nu_{\textrm{-}k}, \nu_{\textrm{-}k}+1]$  for $k \in \Z$ and $\calP_{\Llambda}(-1/2, 1/2]$, respectively. We denote by $\dLambda(\tau_k)$ the derived stack of complexes of finite-dimensional nilpotent $\Pi_\qv$ representations which are flat with respect to $\tau_k$ for $k\in \Z\cup \{1/2\}$ (cf.\ \cite[Construction~\ref*{COHA-Yangian-construction:derived-stack-D_0}]{DPSSV-3}) and by $\dLambda(\tau_k; \bfv)$ its connected component corresponding to $\bfd\in \Z I$.

For any $k\in \Z$, define the derived stack
\begin{align}
	\kdLambdaqv\coloneqq \dLambda(\tau_k) \cap \dLambda(\tau_{\sfrac{1}{2}}) \ .
\end{align}
Then, $\kdLambdaqv$ parametrizes the complexes of finite-dimensional nilpotent $\Pi_\qv$ representations belonging to the category $\calP_{\Llambda}(\nu_{\textrm{-}k}, 1/2]$.

\subsection{Limiting COHAs}\label{subsec:limit-coha-affine-quiver}

Note that $T\coloneqq \C^\ast \times \C^\ast$ acts on $\Lambda_\qv$ and $\dLambda_\qv$. There is a subtorus of $T$ which maps isomorphically onto $A$. For simplicity, we shall denote it by $A$ as well. All the stacks introduced in the previous section admit an $A$-action. We set
\begin{align}
	\coha_J^A\coloneqq \HBMbulletA(\dLambda(\tau_{\sfrac{1}{2}}))\ .
\end{align}
For any $k\in \Z$, define
\begin{align}
	\coha_{J, (k)}^A\coloneqq \HBMbulletA(\kdLambdaqv )\ .
\end{align}
Thus, we have a chain of projections
\begin{align}\label{eq:chain-cohas-biinfinite}
	\cdots \twoheadrightarrow \coha_{J,(1)}^A \twoheadrightarrow \coha_{J,(0)}^A \twoheadrightarrow \coha_{J,(-1)}^A \twoheadrightarrow \cdots \ .
\end{align}
Note that all the maps $\rho_{k,k-1}\colon  \coha_{J,(k)}^A \to  \coha_{J,(k-1)}^A$ occuring in Formula~\eqref{eq:chain-cohas-biinfinite} are induced by the open embeddings $\tensor*[^{(k-1)}]{\dLambda}{_\qv}\to \kdLambdaqv$ and hence $(\N \times \Z I)$-graded.

The following isomorphism is a consequence of the definition of Borel--Moore homology (cf.\ \cite[\S\ref*{foundations-sec:BM-homology}]{DPSSV-1}):
\begin{align}
	\coha_J^A\simeq \lim_k \coha_{J, (k)}^A
\end{align}
with respect to the maps $\rho_{k,k-1}$, where the limit is equipped with the \textit{quasi-compact topology}. This is a topological $(\N \times \Z I)$-graded vector space.

We denote by $\coha_\qv^A$ the \textit{nilpotent cohomological Hall algebra of $\qv$}, whose underlying vector space is the $A$-equivariant Borel--Moore homology $\HBMbulletA(\dLambda_\qv)$ of $\dLambda_\qv$ (see \cite[\S\ref*{COHA-Yangian-sec:nilpotent-quiver-COHA}]{DPSSV-3}). Now, we endow $\coha_J^A$ with an associative algebra structure:
\begin{theorem}\label{thm:coha-surface-as-limit1}
	There is a canonical $(\N \times \Z I)$-graded unital associative algebra structure on $\coha_J^A$ induced by the COHA multiplication on $\coha_\qv^A$.
\end{theorem}

\begin{proof}
	The proof follows by combining Theorem~\ref{thm:slicing} and the same arguments as those in the proof of \cite[Theorem~\ref*{COHA-Yangian-thm:coha-surface-as-limit1}--(\ref*{COHA-Yangian-item:coha-surface-as-limit1.1})]{DPSSV-3}. We include it for completeness.

	We apply the framework developed in \cite[\S\ref*{COHA-Yangian-sec:limiting-COHA}]{DPSSV-3}. For $k\in \N$, set $a_k\coloneqq \nu_{\textrm{-}k}+1$. Then, we have
	\begin{align}
		\lim_{k\to +\infty} a_k= -\frac{1}{2}+1=\frac{1}{2}\eqqcolon a_\infty \ .
	\end{align}
	Let $\Lambda\coloneqq \Z I$ and let $v\colon K_0(\nilpPi) \to \Lambda$ be the map that associates to the K-theory class of a nilpotent finite-dimensional representation of $\Pi_\qv$ its dimension vector.
	Now, since
	\begin{align}\label{eq:equivalence}
		(\tau^{-1}\circ (\R L_{\Llambda})^{-2k}) (\nilpPi) \simeq \calP_{\Llambda}(a_k-1,a_k]
	\end{align}
	for any $k\in \N$ by Theorem~\ref{thm:slicing} and the standard $t$-structure on $\Pi_\qv\Mod$ is open, \cite[Assumption~\ref*{COHA-Yangian-assumption:limiting_2_Segal_stack}]{DPSSV-3} holds.
	Again, thanks to the equivalence~\eqref{eq:equivalence}, \cite[Assumption~\ref*{COHA-Yangian-assumption:quasi-compact_interval}]{DPSSV-3} holds for any $k\in \N$ since it is evidently true for $k=0$: in this case, Harder-Narasimhan strata of the moduli stack $\Lambda_\qv$ are known to be quasi-compact and locally closed.
	Since \cite[Assumption~\ref*{COHA-Yangian-assumption:limiting_CoHA_I}--(\ref*{COHA-Yangian-assumption:limiting_CoHA_I-1})]{DPSSV-3} holds for $k=0$, by the equivalence~\eqref{eq:equivalence} it holds for any $k\in \N$. Moreover, \cite[Assumption~\ref*{COHA-Yangian-assumption:limiting_CoHA_I}--(\ref*{COHA-Yangian-assumption:limiting_CoHA_I-2})]{DPSSV-3} holds for $k=k'=0$, by using again  the equivalence~\eqref{eq:equivalence} and \cite[Corollary~\ref*{COHA-Yangian-cor:assumption-holds}]{DPSSV-3}, we obtain that \cite[Assumption~\ref*{COHA-Yangian-assumption:limiting_CoHA_I}]{DPSSV-3} holds for any $k\in \N$.
	Thus, we can apply \cite[Proposition~\ref*{COHA-Yangian-prop:limiting_CoHA}]{DPSSV-3} and we obtain an $A$-equivariant limiting cohomological Hall algebra
	\begin{align}
		\coha_{\nilpPi, \tau_{1/2}}^A\coloneqq \bigoplus_{\bfd \in \Z I} \lim_{k} \colim_{s \geqslant k} \sfH_\bullet^A\big( \dstackCohps\big( \scrD_0, (a_k-1, a_s]; \bfd \big) \big)
	\end{align}
	as a $\Lambda$-graded vector space, endowed with the quasi-compact topology. Here, $\scrD_0\coloneqq \catDb(\nilpPi)$ and $\dstackCohps\big( \scrD_0, (a_k-1, a_s]; \bfd \big)\coloneqq \dLambda(\tau_k; \bfd)\cap \dLambda(\tau_s; \bfd)$ (see \cite[Construction~\ref*{COHA-Yangian-construction:limiting_COHA-1}]{DPSSV-3} for its definition). Now, \cite[Theorem~\ref*{COHA-Yangian-thm:limiting_vs_limit}]{DPSSV-3} yields
	\begin{align}
		\colim_{s \geqslant k} \sfH_\bullet^A\big( \dstackCohps\big( \scrD_0, (a_k-1, a_s]; \bfd \big) \big) \simeq \sfH_\bullet^A\big( \dstackCohps\big( \scrD_0, (a_k-1, \sfrac{1}{2}]; \bfd \big) \big) = \sfH_\bullet^A\big( \tensor*[^k]{\dLambda}{_\bfd} \big) \ .
	\end{align}
	Therefore, as $\Lambda$-graded vector spaces, $\coha_{\nilpPi, \tau_{1/2}}^A$ is isomorphic to $\coha_J^A$, endowed with the quasi-compact topology.
\end{proof}

\begin{remark}
	Consider the derived moduli stack $\dLambda(\tau_{1/2})$ of complexes of finite-dimensional nilpotent $\Pi_\qv$ representations which are flat with respect to $\tau_{1/2}$. As explained in \cite[\S\ref*{torsion-pairs-sec:COHA-t-structure}]{DPS_Torsion-pairs}, there is a \textit{$2$-Segal space} $\calS_\bullet \dLambda(\tau_{1/2})$ canonically associated to $\dLambda(\tau_{1/2})$. In particular, we have a convolution diagram
	\begin{align}
		\begin{tikzcd}[ampersand replacement=\&]
			\dLambda(\tau_{1/2}) \times \dLambda(\tau_{1/2}) \& \calS_2\dLambda(\tau_{1/2}) \ar{r}{p} \ar{l}[swap]{q}  \& \dLambda(\tau_{1/2})
		\end{tikzcd}\ ,
	\end{align}
	where $\calS_2\dLambda(\tau_{1/2})$ is equivalent the derived stack parametrising distinguished triangles of complexes of finite-dimensional nilpotent $\Pi_\qv$ representations which are flat with respect to $\tau_{1/2}$. Here the maps $p$ and $q$ sends a triangle $E_1\to E_2 \to E_3 \to E_1[1]$ to $E_2$ and $(E_3, E_1)$, respectively. 
	
	As explained in \textit{loc.\ cit.}, the above convolution diagram induces a COHA structure on the equivariant Borel--Moore homology $\HBMbulletA(\dLambda(\tau_{\sfrac{1}{2}}))$ if $q$ is quasi-compact, \textit{finitely connected}\footnote{In the sense of \cite[Definition~\ref*{torsion-pairs-def:modified_classes_of_morphisms}--(\ref*{torsion-pairs-item:finitely-connected})]{DPS_Torsion-pairs}.}, and derived lci, and $p$ is \textit{locally rpas}\footnote{In the sense of \cite[Definition~\ref*{torsion-pairs-def:modified_classes_of_morphisms}--(\ref*{torsion-pairs-item:locally-rpas})]{DPS_Torsion-pairs}.}. Now, the approximation procedure, in the sense of \cite[\S\ref*{COHA-Yangian-sec:stabilization}]{DPSSV-3}, performed in the proof of Theorem~\ref{thm:coha-surface-as-limit1} implies that these properties hold. Thus, we can also define an `intrisic' COHA structure on $\coha_J^A$, which is canonically associated to the $t$-structure $\tau_{1/2}$. 
	
	The equivalence of this `intrisic' COHA with the one given by Theorem~\ref{thm:coha-surface-as-limit1} follows from \cite[Theorem~\ref*{COHA-Yangian-thm:coha-surface-as-limit1}--(\ref*{COHA-Yangian-item:coha-surface-as-limit1.1})]{DPSSV-3}.
\end{remark}

\section{Limits of affine Yangians}\label{sec:limit-affine-Yangian}

In this final section, we use the results of \cite{DPSSV-3} to explicitly describe, in terms of suitable limits of subquotients of affine Yangians, the COHA $\coha_J^A$ of the category $\catCohC(\calA_J)$ of (nilpotent) coherent sheaves on a Kleinian orbifold $\calX_J$. Since there are no interesting cases of Kleinian orbifolds\footnote{besides the two extreme cases $J=\emptyset$ and $J=\{1\}$, which are covered in \cite{DPSSV-3}} when $\qv\simeq A_1^{(1)}$, we assume throughout that $\qv \neq A_1^{(1)}$.

\subsection{Affine Yangians}

\subsubsection{Presentation of the affine Yangian}

We briefly recall the definition of the affine Yangian $\Y_\qv$ relevant to this work. We put $R_T=\Q[\varepsilon^{\pm 1}_1, \varepsilon^{\pm 1}_2]$ and $\hbar\coloneqq \varepsilon_1+\varepsilon_2$. We keep the notation in force regarding the quiver $\qv$. In order to avoid confusion caused by the unconventional choice of $\rootlattice$-grading (see Formula~\ref{eq:simpleclasses}), we will work with $\Z I$-grading and only convert to $\rootlattice$-grading later. To this end, we denote by $\{\epsilon_i\;|\; i \in I\}$ the canonical basis of $\Z I$ and put $\partial\coloneq \epsilon_0 + \sum_i r_i \epsilon_i$.

\begin{definition}\label{def:affine-Yangian}
	Let $\qv$ be an affine ADE quiver, $\qv\neq A_1^{(1)}$. The (affine, two-parameter) Yangian $\Y_{\qv;\, T}$ of $\qv$ is the unital associative $R_T$-algebra generated by $x_{i, \ell}^\pm, h_{i, \ell}$, with $i \in I$ and $\ell \in \N$, subject to the relations
	\begin{itemize}\itemsep0.2cm
		
		\item for any $i, j\in I$ and $r,s\in \N$
		\begin{align}\label{eq:affine-Yangian-Lie-algebra-2}
			\Big[h_{i,r}, h_{j,s}\Big]  & = 0\ ,\\[4pt]  \label{eq:affine-Yangian-Lie-algebra-3}
			\Big[x_{i,r}^{+}, x_{j,s}^{-}\Big] &= \delta_{i,j} h_{i, r+s} \ , \\[4pt] \label{eq:affine-Yangian-4}
			\Big[h_{i,0}, x_{j,r}^\pm\Big]&=\pm a_{i,j}x_{j,r}^\pm\ , \\[4pt]\label{eq:affine-Yangian-5}
			\Big[h_{i, r+1}, x_{j, s}^{\pm}\Big] - \Big[h_{i, r}, x_{j, s+1}^{\pm}\Big]& = \pm \frac{\hbar}{2} a_{i,j} \Big\{h_{i, r}, x_{j, s}^{\pm}\Big\}-m_{i,j}\frac{\varepsilon_1-\varepsilon_2}{2}\Big[h_{i,r},x_{j,s}^{\pm}\Big]\ ,\\[4pt] \label{eq:affine-Yangian-6}
			\Big[x_{i, r+1}^{\pm}, x_{j, s}^{\pm}\Big] - \Big[x_{i, r}^{\pm}, x_{j, s+1}^{\pm}\Big] &= \pm \frac{\hbar}{2} a_{i,j} \Big\{x_{i, r}^{\pm}, x_{j, s}^{\pm}\Big\}-m_{i,j}\frac{\varepsilon_1-\varepsilon_2}{2}\Big[x_{i,r}^{\pm},x_{j,s}^{\pm}\Big] \ ,
		\end{align}
		where
		\begin{align}
			m_{i,j}\coloneqq \begin{cases}
				1 & \text{if } i\to j\in \Omega\ , \\
				-1 & \text{if } j\to i \in \Omega \ , \\
				0 & \text{otherwise} \ .
			\end{cases}
		\end{align}
		
		\item Serre relations: 
		\begin{align}\label{eq:affine-Yangian-Serre}
			\sum_{\sigma \in \frakS_m}\Big[x_{i, r_{\sigma(1)}}^{\pm}, \Big[x_{i, r_{\sigma(2)}}^{\pm}, \Big[\cdots, \Big[x_{i, r_{\sigma(m)}}^{\pm}, x_{j,s}^{\pm}\Big]\cdots\Big]\Big]\Big] = 0 
		\end{align}
		for $i, j\in I$, with $i \neq j$, where $m\coloneqq 1 - a_{i,j}$ and $\frakS_m$ denotes the $m$-th symmetric group. 
	\end{itemize}	
\end{definition}

\begin{notation}
	Identifying $R_T$ with the $T$-equivariant cohomology ring $\Hbullet_T$, we have a surjection $R_T \to \Hbullet_A$. We denote by $\Y_{\qv;\, A}\coloneq \Y_{\qv;\,T} \otimes_{R_T} \Hbullet_A$ the specialization of $\Y_{\qv;\,T}$. 
\end{notation}

\begin{remark} 
		Giving the generators $x^{\pm 1}_{i,\ell}$ and $h_{i,\ell}$ the respective degree $\pm \epsilon_i$ and $0$ induces a $\Z I$-grading on $\Y_{\qv;\,A}$.
	
	The $\rootlattice$-grading conventions here and below differ from those of \cite{DPSSV-3}, by the automorphism $-w_0$. This is in accordance with the non-standard identification between $\Z I$ and $\rootlattice$. 
\end{remark}	

The positive and negative halves $\Y^\pm_{\qv;\,A}$ of $\Y_{\qv,A}$ are defined to be the $A$-subalgebras respectively generated by $\{x^{\pm}_{i,\ell}\,|\, i \in I, \ell \in \N\}$. The loop Cartan subalgebra $\zeroeY_\qv$ is defined to be the subalgebra generated by $\{h_{i,\ell}\,|\, i \in I, \ell \in \N\}$. There is a PBW-type isomorphism
\begin{align}
	\Y^+_{\qv;\,A}\otimes \zeroeY_{\qv;\,A}\otimes \Y^-_{\qv;\,A} \simeq \Y_{\qv;\,A}
\end{align}
(see \cite[Theorem~\ref*{COHA-Yangian-thm:triangular-decomposition-affine}]{DPSSV-3}) and an associated projection map
\begin{align}\label{eq:projection}
	\pr\colon \Y_{\qv;\,A} \longrightarrow \Y^-_{\qv;\,A}\ .
\end{align}

We next describe the classical limit $\varepsilon_1=\varepsilon_2=0$ of $\Y_{\qv;\,T}$. Recall that the \textit{elliptic Lie algebra} $\frakgell$ associated to the Dynkin quiver $\qvfin$ associated to $\qv$ is the universal central extension 
\begin{align}\label{eq:elliptic}
	\frakgell \coloneqq \frakgfin[s^{\pm 1}, t]\oplus K\quad \text{with } K\coloneqq \bigoplus_{\ell\in \N}\Q c_\ell \oplus \bigoplus_{\genfrac{}{}{0pt}{}{\ell\in \N, \; \ell\geqslant 1}{k\in \Z, \; k\neq 0}} \Q c_{k, \ell}
\end{align} 
of the double loop algebra $\frakgfin[s^{\pm 1},t]$. Here, $c_{\ell}, c_{k, \ell}$ are central elements, and the Lie bracket is given by
\begin{align}
	[x\otimes s^k t^\ell, y\otimes s^h t^n]=
	\begin{cases}
		[x, y]\otimes t^{\ell+n}+ k(x, y)\cdot c_{\ell+n}  & \text{if } k+h=0\ ,\\[4pt]
		[x, y]\otimes s^{k+h} t^{\ell+n}+ (kh-\ell n) \cdot (x, y)\cdot c_{m+n, g+k} &\text{if } k+h\neq 0\ ,
	\end{cases}
\end{align}
where $(\,,\,)$ is an invariant nondegenerate pairing on $\frakgfin$. We equip the Lie algebra $\frakgell$ with the $\Z\times\rootlattice$-grading such that 
\begin{align}
	\deg\big(x\otimes s^k t^\ell\big)\coloneqq (-2\ell, \bfd+k\delta) \ , \quad \deg\big(c_{k, \ell}\big)\coloneqq (-2\ell, k\delta) \ , \quad \deg\big(c_{\ell}\big)\coloneqq (-2\ell, 0)\ ,
\end{align}
where $x$ belongs to the root space $(\frakgfin)_{\bfd}, k\in \Z$, and $\ell\in \N$. We'll call the first term of the grading the \textit{horizontal grading} and the second term the \textit{vertical grading}. 

The \textit{negative half} $\fraknell$ of $\frakgell$ is defined as the Lie subalgebra spanned by the homogeneous elements whose horizontal grading belongs to $-\N I\smallsetminus\{0\}$, i.e.
\begin{align}\label{eq:negative-half}
	\fraknell\coloneqq \frakn[t] \oplus K_- \quad\text{where } K_-\coloneqq \bigoplus _{k<0} \Q c_{k, \ell}\ ,
\end{align}
where $\frakn\coloneqq s^{-1}\frakgfin[s^{-1}] \oplus \fraknfin$ is the (standard) \textit{negative nilpotent half} of $\frakg$.

Let us denote by $X_i^\pm$ and $H_i$, with $i=1, \ldots, e$, the Chevalley generators for $\frakgfin$ normalized so that $(X_i^+,X_i^-) = 1$ and $H_i= [X_i^+,X_i^-]$. Let $X_{\pm \varphi}$ be root vectors of $\frakgfin$ for the roots $\pm\varphi$ normalized so that $(X_\varphi,X_{\textrm{-}\varphi}) = 1$, where $\varphi$ is the highest root of $\frakgfin$. Set $H_\varphi\coloneqq [X_\varphi^+,X_\varphi^-]$.
	
\begin{theorem}[{\cite[Propositions~\ref*{COHA-Yangian-prop:classical-limit} and \ref*{COHA-Yangian-prop:iso-s-uce-1}]{DPSSV-3}}]
	The assignment
	\begin{align}
		x_{i, \ell}^\pm &\longmapsto X_i^\pm \otimes t^\ell \quad \text{for } i \in I\;\text{and}\; \ell\in \N\ , \\
		x_{0, \ell}^\pm &\longmapsto X_{\mp\varphi}\otimes t^\ell s^{\pm 1}\quad \text{for } \ell\in \N\ , \\
		h_{i, \ell} &\longmapsto H_i\otimes t^\ell \quad \text{for } i=1, \ldots, e\;\text{and}\; \ell\in \N\ , \\
		h_{0, \ell} &\longmapsto H_\varphi\otimes t^\ell+t^\ell s^{-1} ds \quad \ell\in \N\ .
	\end{align}
	extends to an algebra isomorphism 
	\begin{align}
		\Psi\colon \Y_{\qv;\,T}\otimes_{R_T} \Q \simeq \sfU(\frakgell) \ ,
	\end{align} 
	where the morphism $R_T \to \Q$ is given by $\varepsilon_1, \varepsilon_2 \mapsto 0$. This isomorphism is compatible with negative halves, i.e. $\Psi(\Y_{\qv;\,T}^-\otimes_{R_T} \Q)= \sfU(\fraknell)$.
\end{theorem}

The relation between the nilpotent COHA of the preprojective algebra $\Pi_\qv$ and the affine Yangian is given by the following.
\begin{theorem}[{\cite[Theorem~\ref*{COHA-Yangian-thm:PBW}]{DPSSV-3}}]\label{thm:coha=yangian}
	There is an isomorphism of $\N I$-graded $\Hbullet_A$-algebras 
	\begin{align}\label{eq:Phi}
		\begin{split}
			\Phi\colon\Y_{\qv;\, A}^-&\longrightarrow \cohaqv^{A}\ ,\\
			x^-_{i,\ell} &\longmapsto (z_{i,1})^\ell\cap[\Lambda_{\alpha_i}]
		\end{split}
	\end{align}
	for $i\in I$ and $\ell\in \N$. Here, $z_{i,1}$ is the first Chern class of the tautological bundle on $\Lambda_{\alpha_i}$.
\end{theorem}
When there is no risk of confusion, we will henceforth simply write $\Y_\qv$ for $\Y_{\qv;\ A}$.

\subsubsection{Truncated braid group action}

There is a well-known action of the affine braid group by automorphisms on the affine Yangian, first considered in \cite{GNW18} (cf.\ Formula~(3.15) in \textit{loc.cit.}). It is given by the following formulas:
\begin{align}\label{eq:Ti}
	B_\qv&\longrightarrow \Aut(\Y_\qv) \ , \\
	T_i&\longmapsto \exp\big(\ad\big(x_i^+\big)\big)\circ\exp\big(-\ad\big(x_i^-\big)\big)\circ\exp\big(\ad\big(x_i^+\big)\big)
\end{align}
for $i\in I$. There is a similar (obvious) action of the group $\Gamma$ of diagram automorphisms, and the two actions induce an action of $B_{\sfex}$ on $\Y_\qv$.

Denoting by $\Y_{\qv,\bfd}$ the $\bfd$-weight space of $\Y_{\qv}$, we have $T_i(\Y_{\qv,\bfd})=\Y_{\qv,s_i(\bfd)}$. Obviously, the action of $B_\qv$ does not preserve $\Y^-_\qv$ (rather, it maps $\Y^-_\qv$ isomorphically into a `twisted' negative half of $\Y_\qv$). For any $w\in W_\qv$, we may however define a linear operator $\overline T_w$ as the composition
\begin{align}\label{eq:overlineT}
	\overline T_w\colon
	\begin{tikzcd}[ampersand replacement=\&]
		\Y_\qv^-\arrow{r} \& \Y_\qv\arrow{r}{T_w} \&\Y_\qv\arrow{r}{\pr} \&\Y_\qv^-
	\end{tikzcd}\ ,
\end{align} 
where the last map is the projection \eqref{eq:projection}. We call $\overline T_w$ the \textit{truncated} braid group operator associated to $w$.

Let $B^+_\qv\subset B_\qv$ be the submonoid generated by the elements $T_w$ with $w\in W_\qv$.
\begin{proposition}[{\cite[Propositions~\ref*{COHA-Yangian-prop:truncated-action} and \ref*{COHA-Yangian-prop:truncated-action-affine}]{DPSSV-3}}]
	The assignment $T_w \mapsto \overline T_w$ for $w \in W_\qv$ gives rise to a representation of $B^+_\qv$ on $\Y_\qv^-$, i.e.\ to a morphism of groupoids $B^+_\qv \to \End(\Y_\qv^-)$. It extends to a morphism $B^+_{\sfex} \to \End(\Y_\qv^-)$ where  $B^+_{\sfex} \coloneq B^+_\qv \ltimes \Gamma$.
\end{proposition}

\subsubsection{Quotients of affine Yangians}

In this section, we describe the quotients of $\Y^-_\qv$ which correspond, under the isomorphism $\Phi$, to the quotients $\coha_{J, (k)}^A$, for $k \leq 0$. Recall that we have fixed $\Llambda \in \coweightlatticefin$ and that we are considering the Bridgeland stability condition
\begin{align}
	 Z_\Llambda \coloneqq \Llambda + i\cdot w_0\Lrho \; \in\; \calC_J^0+i\cdot \calC^+\ .
\end{align}
We keep the notations from \S\ref{sec:numerical-tilts}; we set $\nu_{k}\coloneqq t_{2k}$. Through the identification $\Z I \simeq \sfK_0(\nilpPi_\qv)\simeq \rootlattice$ we may view the phase function as a map $\varphi\colon \Z I \to \R$. Put, for $k \leq 0$,
\begin{align}
	\J_{J,(k)}\coloneqq \sum_{ \varphi(\bfd)\leq\nu_{-k}-1} \Y^-_{\qv,\bfd} \Y^-_\qv + \sum_{ \varphi(\bfd)>1/2} \Y^-_{\qv} \Y^-_{\qv,\bfd} \qquad \text{and} \qquad \Y_{J,(k)}\coloneqq \Y^-_\qv/\J_{J,(k)}\ .
\end{align}

Observe that for any $k_1<k_2\leq 0$ there exists a canonical (surjective) quotient map
\begin{align}
	\pi_{k_2,k_1}\colon \Y_{J,(k_2)} \longrightarrow \Y_{J,(k_1)}\ . 
\end{align}

We next summarize the main properties of the quotients $\Y_{J,(k)}$, whose proofs are verbatim the same as in \cite[\S\ref*{COHA-Yangian-subsec:quotients}]{DPSSV-3}.
\begin{theorem}\label{prop:quotients_affine_yangian}
The following holds:
	\begin{enumerate}\itemsep0.2cm
		\item For any $k_1<k_2\leq 0$ we have a commutative diagram
			\begin{align}
			\begin{tikzcd}[ampersand replacement=\&, column sep=large]
				\Y^-_\qv \ar{r}{\overline{T}_{2(k_2-k_1)\Llambda}} \ar{d}{} \& \Y^-_\qv \ar{d}{}\\
				\Y_{J, (k_2)}  \ar{r}{} \&
				\Y_{J, (k_1)} 
			\end{tikzcd}\ ,
		\end{align}
		where the lower horizontal map is an isomorphism, which we denote $T_{2(k_2-k_1)\Llambda}$.
		
		\item \label{item:quotients_affine_yangian_1}
		For any $k \leq 0$, the isomorphism $\Phi$ of Theorem~\ref{thm:coha=yangian} induces an isomorphism of graded vector spaces
		\begin{align}
			\Phi_{J,(k)}\colon \Y^-_{J,(k)} \longrightarrow \coha^A_{J,(k)}\ .
		\end{align}
		
		\item \label{item:quotients_affine_yangian_2}
		For any $k_1<k_2\leq 0$ there is a commutative diagram
		\begin{align}
			\begin{tikzcd}[ampersand replacement=\&, column sep=large]
			\Y_{J,(k_2)} \ar{r}{\pi_{k_2,k_1}} \ar{d}{\Phi_{J,(k_2)}} \& \Y_{J, (k_1)} \ar{d}{\Phi_{J,(k_1)}}\\
			\coha_{J, (k_2)}^A  \ar{r}{\rho_{k_2,k_1}} \&
			\coha_{J, (k_1)}^A 
			\end{tikzcd}
		\end{align}
		
		\item \label{item:quotients_affine_yangian_3}
		For any $k_1<k_2\leq 0$ there is a commutative diagram
		\begin{align}
			\begin{tikzcd}[ampersand replacement=\&, column sep=large]
			\Y_{J,(k_2)} \ar{r}{T_{2(k_2-k_1)\Llambda}} \ar{d}{\Phi_{J,(k_2)}} \& \Y_{J, (k_1)} \ar{d}{\Phi_{J,(k_1)}}\\
			\coha_{J, (k_2)}^A  \ar{r}{\R L_{2(k_2-k_1)\Llambda}} \&
			\coha_{J, (k_1)}^A 
		\end{tikzcd}
		\end{align}
	\end{enumerate}
\end{theorem}

Observe that the projection $\pi_{k_2,k_1}$ preserves the weight, but is not an isomorphism, while the braid operator $T_{2(k_2-k_1)\Llambda}$ is an isomorphism but acts on the weight spaces as the translation $t_{2(k_2-k_1)\Llambda}$.
  
\subsection{Limit}\label{subsec:def-limit-affine-Yangian}

We are now in position to describe the limit COHA $\coha^A_{J}$ as a projective limit of quotients of $\Y^-_\qv$. For an arbitrary $k \in \N$, we define
\begin{align}
	\Y_{J, (k)}\coloneqq T_{\textrm{-}2k\Llambda}(\Y^-_\qv)/T_{\textrm{-}2k\Llambda}(\J_{J,(0)})\ .
\end{align}
We have, by \textit{transport de structure}, an isomorphism $T_{\textrm{-}2\Llambda}\colon \Y_{J, (k)} \to \Y_{J, (k-1)}$ which is $\N$-graded but acts as $t_{\textrm{-}2\Llambda}$ on the weight as well as a restriction map $\pi_{k,k-1}\colon \Y_{J, (k)} \to \Y_{J, (k-1)}$ which is a map of $\N \times \rootlattice$-graded vector spaces. 

Define
\begin{align}
	\Y_J^+ \coloneqq \lim_k \Y_{J, (k)}\ ,
\end{align}
the limit being equipped with the quasi-compact topology again. The following result is now proved in the same way as \cite[Theorem~\ref*{COHA-Yangian-thm:coha-surface-as-limit2}]{DPSSV-3}.
\begin{theorem}\label{thm:coha-surface-as-limit2}
	The following holds:
	\begin{enumerate}\itemsep0.2cm
		\item \label{item:coha-surface-as-limit2.1} There is a canonical $(\N \times \rootlattice)$-graded algebra structure on $\Y_J^+$ induced by the multiplication on $\Y_\qv$.

		\item \label{item:coha-surface-as-limit2.2} There is a canonical $(\N \times \rootlattice)$-graded algebra isomorphism
		\begin{align}
			\begin{tikzcd}[ampersand replacement=\&]
				\Phi_{\Llambda}\colon \coha_J^A\ar{r}{\sim} \& \Y_J^+
			\end{tikzcd}\ ,
		\end{align}
		induced by the maps $\Phi_{(k)}$.
	\end{enumerate}
\end{theorem}

\section{Classical limit of $\coha_J$}\label{sec:classical-limit}

In this final section, we use Theorem~\ref{thm:coha-surface-as-limit2} to compute the classical limit of $\coha^A_J$, and identify it with a twisted positive half of the elliptic Lie algebra $\frakgell$.

As the forgetful map $\HBMbulletA(\dLambda_\qv) \otimes_{\Hbullet_A} \Q \simeq \HBMbullet(\dLambda_\qv)$ is functorial, it induces an isomorphism of algebras $\HBMbulletA(\kdLambdaqv) \otimes_{\Hbullet_A} \Q \simeq \HBMbullet(\kdLambdaqv)$ for any $k$, hence by Theorem~\ref{thm:coha-surface-as-limit1} we get
\begin{align}
	\coha_J^A\otimes_{\Hbullet_A} \Q \simeq \coha_J\ ,
\end{align}
where the right-hand-side is now the non-equivariant cohomological Hall algebra $\coha_J$. In view of \cite[Propositions~\ref*{COHA-Yangian-prop:classical-limit} and \ref*{COHA-Yangian-prop:iso-s-uce-1}]{DPSSV-3}, this gives, for $J=\emptyset$, an isomorphism of algebras
\begin{align}\label{eq:cohaquiver-classical-limit}
	\cohaqv \simeq \sfU(\fraknell)\ ,
\end{align}
where $\fraknell$ is given by Formula~\eqref{eq:negative-half}. Since Theorem~\ref{thm:coha-surface-as-limit1} works regardless of torus actions, the non-equivariant versions of Theorems~\ref{thm:coha-surface-as-limit1} and \ref{thm:coha-surface-as-limit2} yield an isomorphism of topological algebras
\begin{align}\label{eq:cohasurface-classical-limit}
	\coha_J \simeq \lim_k \Y_{J, (k)}\otimes_{\Hbullet_A} \Q\ .
\end{align}

We now use the isomorphisms~\eqref{eq:cohaquiver-classical-limit} and \eqref{eq:cohasurface-classical-limit} to obtain an explicit description of $\coha_J$. From \cite[\S\ref*{COHA-Yangian-sec:quotients-Yangians}]{DPSSV-3}, we have
\begin{align}
	\Y_{J, (0)}\otimes_{\sfH^\bullet_A} \Q \simeq \sfU((\fraknell)_{J,(0)})\quad \text{where} \quad (\fraknell)_{J,(0)}\coloneqq\bigoplus_{\beta \in \Delta_{J, (0)}} (\frakn)_{\beta}[t] \oplus K_-\ ,
\end{align}
with
\begin{align}
	\Delta_{J, (0)}\coloneqq \Big\{\alpha +n\delta \; \Big\vert\; (\Llambda, \alpha)\geq 0;\ (\alpha\in \Delta^-_\sff,\ n=0) \;\text{or}\; (\alpha \in \Delta_\sff \cup\{0\},\ n<0) \Big\}\ .
\end{align}
Now, set $\Jcomp\coloneqq I_\sff\smallsetminus J$ and
\begin{align}
	\Delta_{\Jcomp}\coloneqq \Delta_\sff \cap \bigoplus_{i\in \Jcomp} \Z \alpha_i\quad \text{and} \quad \Delta_{\Jcomp}^\pm\coloneqq \Delta_\sff^\pm \cap \bigoplus_{i\in \Jcomp} \Z \alpha_i\ .
\end{align}
We have $(\Llambda,\alpha) \geq 0$ if and only $\alpha \in \Delta_{\Jcomp}$ or $\alpha \in \Delta_\sff^+ \smallsetminus \Delta^+_{\Jcomp}$. Thus we obtain that
\begin{align}
	\Delta_{J, (0)} = \Big\{\alpha +n\delta \; \Big\vert\; (\alpha\in \Delta^-_{\Jcomp},\  n=0)\; \text{or}\; (\alpha \in \Delta_{\Jcomp}\cup (\Delta^+_{\sff} \smallsetminus \Delta^+_{\Jcomp}) \cup \{0\},\  n<0) \Big\} \ .
\end{align}

The subquotient $\Y_{J, (k)}\otimes_{\sfH^\bullet_A} \Q$ of $\sfU(\frakgell)$ is obtained from $\Y_{J, (0)}\otimes_{\sfH^\bullet_A} \C$ by applying the automorphism $T_{\textrm{-}2k\Llambda}$. It follows that
\begin{align}
	\Y_{J, (k)}\otimes_{\sfH^\bullet_A} \Q \simeq \sfU((\fraknell)_{J,(k)}) \quad \text{where} \quad (\fraknell)_{J,(k)}\coloneqq \bigoplus_{\beta \in \Delta_{J, (k)}} (\frakn)_{\beta}[t] \oplus K_-\ ,
\end{align}
with
\begin{multline}
	\Delta_{J, (k)}\coloneqq \Big\{\alpha +n\delta \; \Big\vert\; (\alpha\in \Delta^-_{\Jcomp},\ n=0) \;\text{or}\;(\alpha\in \Delta^+_{\Jcomp}\cup \{0\},\ n<0)\;\text{or}\;(\alpha\in \Delta^+_\sff \smallsetminus \Delta^+_{\Jcomp},\  n<2k(\Llambda, \beta))\Big\} \ .
\end{multline}

Define
\begin{align}
	\frakr_{(k)}\coloneqq\bigoplus_{\beta \in \Delta_{J, (k)} \smallsetminus \Delta_{J, (k-1)}} (\frakn)_{\beta}[t]
\end{align}
so that for $k \geq 1$ we have $(\fraknell)_{J,(k)}=(\fraknell)_{J,(k-1)} \oplus \frakr_{(k)}$. The PBW isomorphism
\begin{align}
	\sfU((\fraknell)_{J,(k)}) \simeq  \sfU(\frakr_{(k)})\otimes \sfU((\fraknell)_{J,(k-1)})
\end{align}
gives rise to a projection morphism $\sfU((\fraknell)_{J,(k)}) \to \sfU(\frakn_{J,(k-1)})$, and we have
\begin{align}
	\coha_J \simeq \Y_J^+\otimes_{\sfH^\bullet_A} \Q\simeq \lim_k \sfU((\fraknell)_{J,(k)})\ .
\end{align}

Note that, as subsets of $\Delta$, we have $\Delta_{J,(0)} \subset \Delta_{J,(1)} \subset \cdots$ and, setting $\Delta_J\coloneqq\bigcup_k \Delta_{J,(k)}$ we get
\begin{multline}
	\Big\{\alpha +n\delta \; \Big\vert\; (\alpha\in \Delta^-_{\Jcomp},\ n=0) \;\text{or}\;(\alpha\in \Delta^+_{\Jcomp}\cup \{0\},\ n<0)\;\text{or}\;(\alpha\in \Delta^+_\sff \smallsetminus \Delta^+_{\Jcomp},\  n\in \Z, )\Big\} \ .
\end{multline}

\begin{remark}
	Note that $\Delta_J$ corresponds exactly to the set determined by Jacobsen and Kac in \cite[Formula(1.5)]{Jakobsen-Kac-Borel-II} for $\rsv$, in their notation, equals, $J$. As proved in \cite{Jakobsen-Kac-Borel} (see \cite[Proposition~1.7]{Jakobsen-Kac-Borel-II}), any set of positive roots of $\Delta$ is $W\times \{\pm 1\}$-conjugate to one of the sets $\Delta_J$.
\end{remark}

In terms of the isomorphism between $\frakgell$ and the universal central extension of $\frakgfin[s^{\pm 1}, t]$, we have
\begin{align}\label{eq:Ln'}
	\fraknellJ^+\coloneqq& \bigoplus_{\beta \in \Delta_{J}} (\frakn)_\beta[t] \oplus K_-\\
	=& \bigoplus_{\alpha \in \Delta^+_\sff \smallsetminus \Delta^+_{\Jcomp}}\hspace{-5pt}\frakg_{\alpha}[s^{\pm 1},t] \ \oplus\   \frakn_{\Jcomp}[t] \ +\  s^{-1}\frakh[s^{-1},t]\  \oplus \ K_-\ ,
\end{align}
where $\mathfrak{l}_{\Jcomp}$ is the Levi subalgebra of $\frakgfin$ corresponding to $\Jcomp$ and $\frakn_{\Jcomp}\coloneqq\frakn_{\mathfrak{l}_{\Jcomp}}^- \oplus s^{-1} \mathfrak{l}_{\Jcomp}[s^{-1}]$ is the standard negative half of the affinization of $\mathfrak{l}_{\Jcomp}$.
Note that $ \frakn_{\Jcomp} \cap s^{-1}\frakh[s^{-1}] = s^{-1}\frakh_{\Jcomp}[s^{-1}]$, hence the sign $+$ instead of $\oplus$ in the above formula.

Let $\widehat{\sfU}(\fraknellJ^+)$ be the completion, in the sense of \cite[Lemma~\ref*{COHA-Yangian-lem:completion-graded-Lie-algebra}]{DPSSV-3}, of $\sfU(\fraknellJ^+)$, with respect to the slope function $\mu_{\Llambda}$. Recall that $\fraknellJ^+$ is introduced in Formula~\eqref{eq:Ln'}.
\begin{theorem}\label{thm:coha-surface-as-limit3}
	There is a canonical isomorphism of complete topological algebras
	\begin{align}
		\begin{tikzcd}[ampersand replacement=\&]
			\Phi_J\colon\coha_J \ar{r}{\sim}\& \widehat{\sfU}(\fraknellJ^+)
		\end{tikzcd}\ .
	\end{align}
\end{theorem}

\begin{remark}
	The above computation also allows one to understand the classical limit, in the sense of the standard filtration, of $\coha_J^A$. More precisely, the standard filtration of $\Y_\qv$ induces one on both $\Y^-_\qv$ and the quotients $\Y_{J, (2k)}$ for $k \geq 0$. The transition map $\Y_{J, (0)} \to \Y_{J, (1)}$ is compatible with this filtration because the braid operators $T_i$ and and their truncated versions $\overline{T}_i$ are, by \cite[Formula~(\ref*{COHA-Yangian-eq:Ti})]{DPSSV-3}, compatible with this filtration. This induces 
	a filtration on $\coha_J$. We then have an isomorphism
	\begin{align}
		\begin{tikzcd}[ampersand replacement=\&]
			\gr\, \Phi_{\Llambda}\colon\gr\coha_J^A  \ar{r}{\sim}\& \gr \Y_J^+\simeq \lim_k\gr \Y_{J, (k)} \simeq \lim_k \sfU((\fraknell)_{\Llambda,(k)}\big)\otimes \Hbullet_A \simeq \widehat{\sfU}(\fraknellLlambda^+)\otimes \Hbullet_A
		\end{tikzcd}\ .
	\end{align}
\end{remark}

\bigskip


\newcommand{\etalchar}[1]{$^{#1}$}
\providecommand{\noopsort}[1]{}
\providecommand{\bysame}{\leavevmode\hbox to3em{\hrulefill}\thinspace}
\providecommand{\href}[2]{#2}

\end{document}